\documentclass[a4paper,11pt]{article}
\usepackage[T2A]{fontenc}
\usepackage[utf8x]{inputenc}
\usepackage[english]{babel}

\usepackage{latexsym}
\usepackage{geometry}
\usepackage{graphicx}
\usepackage{amsfonts}
\usepackage{amsmath}
\usepackage{amsthm}
\usepackage{amssymb}
\usepackage{enumerate}
\usepackage{mathtools}
\usepackage[noadjust]{cite}
\usepackage{centernot}
\usepackage{indentfirst}
\usepackage{xcolor}
\usepackage{float}
\usepackage{listings}
\usepackage{hyperref}

\newtheorem{theorem}{Theorem}[section]
\newtheorem{corollary}[theorem]{Corollary}
\newtheorem{proposition}[theorem]{Proposition}
\newtheorem{lemma}[theorem]{Lemma}
\theoremstyle{definition}
\newtheorem{definition}[theorem]{Definition}
\newtheorem{notation}[theorem]{Notation}
\newtheorem{example}[theorem]{Example}
\newtheorem{remark}[theorem]{Remark}

\numberwithin{equation}{section}
\allowdisplaybreaks

\def\A{\mathcal{A}}
\def\R{\mathbb{R}}
\def\C{\mathbb{C}}
\def\F{\mathbb{F}}
\def\Okubo{\mathcal{O}}

\DeclareMathOperator{\n}{n}
\DeclareMathOperator{\s}{s}
\DeclareMathOperator{\tr}{tr}
\DeclareMathOperator{\spn}{span}
\DeclareMathOperator{\D}{d}
\DeclareMathOperator{\diam}{diam}
\DeclareMathOperator{\Ann}{Ann}
\DeclareMathOperator{\Aut}{Aut}
\DeclareMathOperator{\Orb}{Orb}
\DeclareMathOperator{\Fix}{Fix}
\DeclareMathOperator{\chrs}{char}
\DeclareMathOperator{\Jord}{J}

\newcommand{\frsl}{{\mathfrak{sl}}}
\newcommand{\id}{\mathrm{id}}

\providecommand{\keywords}[1]{\textbf{Keywords:} #1}
\providecommand{\msc}[1]{\textbf{MSC 2020:} #1}

\newcommand\freefootnote[1]{%
    \bgroup
    \renewcommand\thefootnote{\fnsymbol{footnote}}%
    \renewcommand\thempfootnote{\fnsymbol{mpfootnote}}%
    \footnotetext[0]{#1}%
    \egroup
}

\begin{document}

\title{On orthogonality graphs of Okubo algebras}
\author{
Danil Pavlinov$^{a,b}$ and Svetlana Zhilina$^{a,b}$
}
\date{\small \em
$^a$Department of Mathematics and Mechanics, Lomonosov Moscow State\\ University, Moscow, 119991, Russia\\
$^b$Moscow Center of Fundamental and Applied Mathematics, Moscow, 119991, Russia
}

\maketitle

\begin{abstract}
The orthogonality graph of an Okubo algebra with isotropic norm over an arbitrary field $\F$ is considered. Its connected components are described, and their diameters are computed. It is shown that there exist at most two shortest paths between any pair of vertices, and the conditions under which the shortest path is unique are determined.
\end{abstract}


\msc{05C25, 17A75.}

\freefootnote{The authors' research was supported by the Ministry of Science and Higher Education of the Russian Federation as part of the program of the Moscow Center of Fundamental and Applied Mathematics under agreement No. 075-15-2025-345.}

\freefootnote{Email addresses: \texttt{pavlinov.d.aa@gmail.com} (Danil Pavlinov), \texttt{s.a.zhilina@gmail.com} (Svetlana Zhilina)}
	
	\section{Introduction}

    One of the convenient methods for analyzing binary relations on algebraic structures is to study the corresponding relation graphs. For an arbitrary (possibly non-unital, noncommutative, and non-associative) algebra~$\A$ over a field~$\F$, of greatest interest are the commuting, zero divisor, and orthogonality graphs, where the elements $a, b \in \A$ are called {\em orthogonal} if $ab = ba = 0$.
    
    In the current paper we focus on zero divisor graphs and orthogonality graphs. Clearly, the structure of these graphs is nontrivial only if $Z(\A)$ and $Z_{LR}(\A)$ are respectively nonempty, where $Z(\A)$ denotes the set of all zero divisors (left, right, or two-sided) in~$\A$, and $Z_{LR}(\A)$ denotes the set of two-sided zero divisors in~$\A$.

    \begin{notation}
    For any subset $X$ of a linear space $V$ over~$\F$ we denote the set of lines passing through nonzero elements of~$X$ by
    \[
    \mathbb{P}(X) = \{ [x] = \mathbb{F} x \mid x \in X \setminus \{ 0 \} \}.
    \]
    \end{notation}
	
	\begin{definition} \label{definition:graphs}
		Let $\A$ be an arbitrary algebra. We define the following relation graphs of~$\A$:
		\begin{itemize}
			\item 
            The {\em orthogonality graph} $\Gamma_O(\A)$: its vertices correspond to the elements of $\mathbb{P}(Z_{LR}(\A))$, and two distinct vertices $[a]$ and $[b]$ are adjacent (denoted by $[a] \leftrightarrow [b]$) if and only if $ab=ba=0$.
            \item
            The {\em directed zero divisor graph} $\Gamma_Z(\A)$: its vertices correspond to the elements of $\mathbb{P}(Z(\A))$, and two distinct vertices $[a]$ and $[b]$ are connected by a directed edge (denoted by $[a] \to [b]$) if and only if $ab=0$.
		\end{itemize}
	\end{definition}

    Henceforth, when speaking about the vertices of the graphs defined above, we will not distinguish between a nonzero element $a$ and a line $[a] = \F a$ passing through it.
    Recall that, in a directed or an undirected graph $\Gamma$, $\D(a,b)$ denotes the {\em distance} between two vertices $a$ and $b$, and $\diam(\Gamma) = \sup\limits_{a,b \in \Gamma} \D(a,b)$ is the {\em diameter} of $\Gamma$.

    \begin{remark}
        In the early definitions of the zero divisor and orthogonality graphs, their vertices were assumed to be the elements of $Z(\A)$ or $Z_{LR}(\A)$, respectively, see~\cite{Guterman-Markova_Orthogonality, Redmond}. Later the notion of the compressed zero divisor graph was introduced, where the vertices correspond to the equivalence classes of elements with the same annihilators. In the case of finite associative rings, such graphs are systematically studied by Afanas'ev and Monastyreva~\cite{Afanasev-Monastyreva}. However, this definition proves to be too strict and implicitly utilizes additional knowledge about the internal structure of the algebra or the ring under consideration. We adhere to an intermediate option, where the extra vertices are deleted, and the graph is not overcrowded, but no specific knowledge about the algebra is needed. Definitions of this type are already widely used in the context of Birkhoff--James orthogonality graphs of normed spaces~\cite{Arambasic23, Guterman}.  
    \end{remark}

    Among the directions where relation graphs find particularly important applications, we mention two intertwining questions: the problem of classification of relation preserving mappings and the isomorphism problem, that is, exploring the connection between the isomorphism of algebraic structures and the isomorphism of corresponding relation graphs. For instance, in~\cite{Arambasic23} both these questions were answered for Birkhoff--James orthogonality graphs of smooth normed spaces, and in~\cite{Dolinar20} the isomorphism problem of finite-dimensional formally real simple Jordan algebras was completely solved in terms of graphs induced by Jordan orthogonality. Note also that linear preserver problems are closely connected to Jordan homomorphisms, see a recent survey by Bre\v{s}ar and Zelmanov~\cite{Bresar}.

    The aim of our study is to investigate relation graphs of an important class of non-associative algebras, namely, Okubo algebras. Okubo algebras were first defined in Okubo's paper~\cite{Oku78} and have subsequently been extensively studied in numerous papers by Okubo and Osborn \cite{Oku78D, Oku78, OO81b}, as well as by Elduque \cite{Eld97, Eld99, Eld09, Eld15, Eld18, Eld23}, who also co-authored with Myung \cite{EM91, EM93} and P\'erez \cite{EM04, EP96}. According to~\cite[Theorem~2.9]{Eld18}, Okubo algebras, along with forms of para-Hurwitz algebras, constitute the class of symmetric composition algebras. Furthermore, any finite-dimensional flexible composition algebra is either unital (i.e., it is a Hurwitz algebra) or symmetric, see~\cite[Theorem~3.2]{EM04}.
    
    In our previous paper~\cite{Zhilina-Pavlinov} commuting graphs of Okubo algebras have been considered. We have particularly focused on the case of the pseudo-octonion algebra over a field~$\F$, $\chrs \F \neq 3$, which contains a primitive cube root of unity, and also on the case of the real division Okubo algebra --- the so-called real pseudo-octonions.     Besides, by~\cite[Corollary~4.6]{Zhilina-Pavlinov}, the distance between any two idempotents in the commuting graph of an arbitrary Okubo algebra is at most~$2$, that is, the intersection of their centralizers is nonzero.
    
    In the current paper we concentrate on orthogonality and zero divisor graphs of Okubo algebras. Our main result is Theorem~\ref{theorem:graph-of-orthogonality} which contains the full description of all connected components in the orthogonality graph of an arbitrary Okubo algebra with isotropic norm. The number of connected components, their structure and diameters may vary depending on the numerical characteristics of the underlying field and the algebra itself. However, the orthogonality graph of an Okubo algebra is always either geodetic or bigeodetic, that is, either the shortest path between any pair of vertices is unique, or there are at most two such paths.

    We remark that our approach to the study of Okubo algebras leads to deeper understanding of their idempotents and the role played by zero divisors satisfying the condition $\n(x,x*x) = 0$. In particular, in the case of the split Okubo algebra over a field of characteristic~$3$, Corollary~\ref{corollary:two-types-char-three} extends the characterization of singular and quadratic idempotents obtained in~\cite[Proposition~8.7]{Eld18} to terms of the orthogonalizers of their corresponding zero divisors, and Corollary~\ref{corollary:unique-path-three} reveals how these zero divisors interact in the orthogonality graph.
    
    The structure of this paper is as follows: Section~2 contains main statements about Okubo algebras that are used throughout the text. In particular, in Subsection~2.1 we present the construction of Okubo algebras over fields of characteristic not~$3$ via pseudo-octonion algebras. In Subsection~2.2 the multiplication table of an arbitrary Okubo algebra with isotropic norm is given, and the special case when the algebra under consideration contains nonzero idempotents is also considered. Subsection~2.3 describes the explicit form of left and right annihilators of zero divisors in Okubo algebras, as well as their intersections, which was obtained by Matzri and Vishne~\cite{Matzri}.
    These results immediately imply Theorem~\ref{theorem:graph-of-zero-divisors-main} which states that the zero divisor graph of an arbitrary Okubo algebra with isotropic norm is connected, and its diameter equals~$2$.
    
    Section~3 is devoted to orthogonality graphs of Okubo algebras with isotropic norm over arbitrary fields. Proposition~\ref{proposition:connectivity-component-with-type3-elem} states that any zero divisor $x$ satisfying $\n(x, x*x) \neq 0$ generates a connected component of diameter~$1$ which consists of only two vertices. If an Okubo algebra contains a nonzero element $x$ such that $\n(x,x*x) = 0$, then it necessarily has a nonzero idempotent~\cite[(34.10)]{KMRT98}. Theorem~\ref{theorem:graph-of-orthogonality} describes the connected components of the orthogonality graph which are formed by such elements. Their number and diameters depend on the characteristic of the field~$\F$, the existence of a primitive cube root of unity~$\omega$ in~$\F$, and the parameter~$\beta$ which determines the multiplication table~\ref{table:okubo-algebra-isotropic} of an Okubo algebra with isotropic norm and nonzero idempotents. More specifically, if $\chrs \F \neq 3$ and $\omega \notin \F$, then these elements form several connected components, and each of them is a star graph, cf. Lemma~\ref{lemma:star-graph}. Otherwise, all elements such that $\n(x,x*x) = 0$ belong to the same connected component whose diameter equals~$5$ if the Okubo algebra is split, and~$3$ otherwise, see Lemmas~\ref{lemma:diameter-five} and~\ref{lemma:orthogonality-characteristic-three} and Example~\ref{example:distance-five}.

    In Section~4 we study the number of shortest paths between an arbitrary pair of vertices in the orthogonality graph. Corollary~\ref{corollary:unique-path} states that, in the case of the split Okubo algebra over a field~$\F$, $\chrs \F \neq 3$ and $\omega \in \F$, there exist at most two shortest paths between any two vertices, and completely describes the conditions under which the shortest path is unique. According to Proposition~\ref{proposition:unique-path} and Corollary~\ref{corollary:unique-path-three}, in the orthogonality graphs of all other Okubo algebras with isotropic norm, all shortest paths are unique. The most interesting case is the split Okubo algebra over a field of characteristic~$3$, since in this case zero divisors satisfying $x*x = 0$ turn out to be related to quadratic and singular idempotents, and thus they are naturally divided into two types, see Corollary~\ref{corollary:two-types-char-three}. The proof of this statement is based on the construction of the split Okubo algebra via a Petersson algebra associated to the Zorn vector-matrix algebra equipped with an automorphism of order~$3$.
    
    In Section~5 we establish a connection between the orthogonality graphs of Okubo algebras over fields which contain the primitive cube root of unity~$\omega$ and the orthogonality graphs of matrix rings. In particular, it is shown that the connected component that consists of the elements satisfying $\n(x,x*x) = 0$ is isomorphic to the subgraph of the orthogonality graph of $3 \times 3$ matrices on the set of nilpotent matrices, which was studied in detail by Bakhadly, Guterman, and Markova~\cite{Guterman-Markova_Orthogonality}. This gives an alternative proof of Theorem~\ref{theorem:graph-of-orthogonality} in the case when $\chrs \F \neq 3$ and $\omega \in \F$. On the other hand, the results of Section~4 show that the orthogonality graph of nilpotent $3 \times 3$ matrices over such a field~$\F$ is bigeodetic, see Corollary~\ref{corollary:nilpotent-bigeodetic}.

    In the Appendix~\ref{section:appexdix-program} we present a simple Wolfram Mathematica program for multiplication and norm computation in an arbitrary Okubo algebra with isotropic norm. It was used to verify all numerical calculations in this paper. 

    \section{Okubo algebras and their properties}
	
	\subsection{Construction of Okubo algebras}
	
	We use the definition of an Okubo algebra given in \cite[Section~1]{Eld99}. Assume first that the characteristic of the field~$\F$ is distinct from~$3$, and~$\F$ contains a primitive cube root of unity $\omega$. In particular, the second condition is satisfied if~$\F$ is algebraically closed. We set 
    \[
    \mu = \frac{1-\omega^2}{3}.
    \]
    Let $\frsl_3(\F)$ denote the Lie algebra of traceless $3 \times 3$ matrices over~$\F$. Define a new non-associative multiplication ``\(*\)'' on $\frsl_3(\F)$ by
    \begin{equation} \label{equation:product}
    x * y = \mu x y + (1-\mu) y x - \frac{\tr(xy)}{3} I,
    \end{equation}
    where $I \in M_3(\F)$ is the identity matrix. The resulting eight-dimensional algebra, denoted by $P_8(\F)$, is called the {\em pseudo-octonion algebra} over~$\F$.
    
    Consider then the quadratic form
    \begin{equation} \label{equation:norm-definition}
    \n(x) = \frac{1}{6}\tr(x^2)
    \end{equation}
    on the algebra $P_8(\F)$. As noted in \cite[p.~101]{Eld99}, this definition makes sense even when $\chrs \F = 2$. Indeed, by \cite[p.~1027, Lemma]{Fau88}, for any matrix $x \in M_3(\F)$ the identity
    \[
    \tr(x^2) = (\tr x)^2 - 2\s(x)
    \]
    holds, where $\s(x)$ is the quadratic form appearing as the coefficient of the characteristic polynomial
    \begin{equation} \label{equation:ch}
    \det(\lambda I - x) = \lambda^3 - \tr(x)\lambda^2 + \s(x)\lambda - \det(x).
    \end{equation}
    Thus $\tr(x^2)$ can be ``divided by 2'' for any $x \in \frsl_3(\F)$, and
    \begin{equation} \label{equation:norm-division-by-3}
    \n(x) = -\frac{1}{3}\s(x).
    \end{equation}
    
    \begin{definition}
    Let $(\A, +, *)$ be an arbitrary (possibly non-unital and non-associative) algebra over a field~$\F$. Assume that $\A$ is equipped with a strictly nondegenerate quadratic form $\n(\cdot)$, i.e., the associated symmetric bilinear form 
    \[
    \n(a,b) = \n(a+b) - \n(a) - \n(b)
    \]
    is nondegenerate on~$\A$. Then $\A$ is called a {\em composition algebra} if the quadratic form $\n(\cdot)$ admits composition, i.e., $\n(a*b) = \n(a)\n(b)$ for all $a,b \in \A$. Here $\n(\cdot)$ is called the {\em norm}.
    \end{definition}
    
    \begin{definition} \label{definition:symmetric-composition}
    A composition algebra $(\A,*,\n)$ is called {\em symmetric} if
    \begin{equation} \label{equation:symmetric}
    \n(x*y,z) = \n(x,y*z)
    \end{equation}
    for all $x,y,z \in \A$. 
    \end{definition}
    By \cite[Lemma~1.1]{EM93}, the algebra $P_8(\F)$ is a symmetric composition algebra. Note that Eq.~\eqref{equation:norm-definition} and the identity $\tr(xy) = \tr(yx)$ imply that
    \begin{equation} \label{equation:bilinear-form}
        \n(x,y) = \dfrac{1}{3}\tr(xy).
    \end{equation}
    
    \begin{theorem}[{\cite[Theorem~1]{Oku78D}}, {\cite[(34.1)]{KMRT98}}] \label{theorem:symmetric-flexible}
    Let $(\A,*,\n)$ be a composition algebra. The following conditions are equivalent:
    \begin{enumerate}[\rm (1)]
    \item $\A$ is symmetric;
    \item $(x*y)*x = x*(y*x) = \n(x)\,y$ for all $x,y \in \A$.
    \end{enumerate}
    \end{theorem}
    
    Any algebra~$\A$ which satisfies condition (2) of the theorem above is {\em flexible}. Namely, for all $x,y \in \A$ the flexibility identity holds:
    \[
    x*(y*x) = (x*y)*x.
    \]
    
    By linearizing the equality in Theorem~\ref{theorem:symmetric-flexible}(2), we obtain the identity
    \begin{equation} \label{equation:lin-of-symm}
    (x*y)*z + (z*y)*x = x*(y*z) + z*(y*x) = \n(x,z)y
    \end{equation}
    which will be heavily used throughout most of the work.
    
    Now assume that the field~$\F$ is not necessarily algebraically closed, and let $\overline{\F}$ be its algebraic closure. An algebra~$\A$ over~$\F$, $\chrs \F \neq 3$, is called an {\em Okubo algebra} if it is an~$\F$-form of $P_8(\overline{\F})$, i.e., 
    \(
    \A \otimes_{\F} \overline{\F} \;\cong\; P_8(\overline{\F}).
    \)
    
    \begin{example}
    The first Okubo algebras to be constructed \cite{Oku78} were the complex algebra $P_8(\C)$ and its real form $\widetilde{P}_8(\R)$, the latter being defined as the set of traceless Hermitian matrices:
    \[
    \widetilde{P}_8(\R) = \big\{ x \in M_3(\C) \mid x^* = x, \: \tr x = 0\big\}.
    \]
    The algebra $\widetilde{P}_8(\R)$ is called the {\em real pseudo-octonion algebra}.
    \end{example}
    
    If $\chrs \F = 3$, then a different approach is needed to define an Okubo algebra. It was first defined in~\cite{OO81b} by presenting an explicit multiplication table, and in~\cite{EP96} a more general approach was offered, based on the ideas of Petersson's construction~\cite{Pet69}. In any case, it is universally true that any Okubo algebra over an arbitrary field~$\F$ is a symmetric composition algebra. We will use $\Okubo$ to denote an arbitrary Okubo algebra.
    
    Note that any Okubo algebra $\Okubo$ satisfies the relation
    \begin{equation} \label{equation:xxxx}
    (x*x)*(x*x) = \n(x,x*x)\, x - \n(x)\,x*x
    \end{equation}
    for all $x \in \Okubo$, see \cite[(34.3)]{KMRT98}. Together with Theorem~\ref{theorem:symmetric-flexible}(2), this implies that the subalgebra generated by an element $x$ coincides with the linear span of $x$ and $x*x$.

    \subsection{Multiplication table of Okubo algebras with isotropic norm}

    A special role in the study of Okubo algebras is played by algebras of the form $\Okubo_{\alpha,\beta}$. Given nonzero $\alpha, \beta \in \F$, the multiplication table of $\Okubo_{\alpha,\beta}$ is described by Table~\ref{table:okubo-algebra-isotropic}, while basis elements satisfy the conditions~\cite[pp. 3-4]{Eld23}:
    \begin{align}
        &\n(z_{i, j}) = 0, \label{equation:norm-of-basis}\\
        &\n(z_{i, j}, z_{i', j'}) = 
        \begin{cases} 
        \alpha^{(i+i')/3} \beta^{(j+j')/3} & \text{if } i + i' \equiv  j + j' \equiv 0 \: (\text{mod } 3), \\
        0, & \text{otherwise},
        \end{cases} \label{equation:scpr-of-basis}
    \end{align}
    where $0 \leq i, j, i', j' \leq 2$ and $(i, j) \neq (0,0) \neq (i', j').$

    \begin{table}[H]
    \centering
    $
    \begin{array}{|c||c|c|c|c|c|c|c|c|}
    \hline
    \vphantom{\Big|} * & z_{1,0}  & z_{2,0} & z_{0,1} & z_{0,2}&  z_{1,1}& z_{2,2} & z_{1,2} & z_{2,1} \\\hline\hline
    \vphantom{\Big|} z_{1,0} & z_{2,0} & 0 & -z_{1,1} & 0 &-z_{2,1} & 0 & 0 & -\alpha z_{0,1}\\\hline
    \vphantom{\Big|} z_{2,0} & 0 & \alpha z_{1,0} & 0 & -z_{2,2} & 0 & -\alpha z_{1,2} & -\alpha z_{0,2} & 0\\\hline
    \vphantom{\Big|} z_{0,1} & 0 & -z_{2,1} & z_{0,2} & 0 & 0 & -\beta z_{2,0} & 0 & - z_{2,2}\\\hline
    \vphantom{\Big|} z_{0,2} & -z_{1,2} & 0 & 0 & \beta z_{0,1} & -\beta z_{1,0} & 0 & -\beta z_{1,1} & 0\\\hline
    \vphantom{\Big|} z_{1,1} & 0 & -\alpha z_{0,1} & -z_{1,2}& 0 & z_{2,2} & 0 & -\beta z_{2,0} & 0\\\hline
    \vphantom{\Big|} z_{2,2} & -\alpha z_{0,2} & 0 & 0 & -\beta z_{2,1} & 0 & \alpha\beta z_{1,1} & 0 & -\alpha\beta z_{1,0}\\\hline
    \vphantom{\Big|} z_{1,2} & -z_{2,2} & 0 & -\beta z_{1,0}& 0 & 0 & -\alpha\beta z_{0,1} & \beta z_{2,1} & 0\\\hline
    \vphantom{\Big|} z_{2,1} & 0 & -\alpha z_{1,1} & 0 & -\beta z_{2,0} & -\alpha z_{0,2} & 0 & 0 & \alpha z_{1,2}\\\hline
    \end{array}
    $
    \caption{\label{table:okubo-algebra-isotropic} Multiplication table of the Okubo algebra $\Okubo_{\alpha,\beta}$.}
    \end{table}

    \begin{definition}
    The Okubo algebra $\Okubo_{1,1}$ is called the {\em split} Okubo algebra.
    \end{definition}

    \begin{proposition}[{\cite[pp.~3--4]{Eld23}}] \label{proposition:split-pseudo-octonion}
        Let $\chrs \F \neq 3$, and $\F$ contains a primitive cube root of unity~$\omega$. Then $\Okubo_{1,1}$ is isomorphic to the pseudo-octonion algebra $P_8(\F)$.
    \end{proposition}
    
    If the field~$\F$ is algebraically closed, then $\Okubo_{1,1}$ is the unique Okubo algebra over~$\F$ up to isomorphism.
    Obviously, the norm $\n(\cdot)$ on an Okubo algebra over an algebraically closed field~$\F$ is {\em isotropic}, that is, there exists a nonzero element $x \in \Okubo$ such that $\n(x) = 0$. According to the following theorem, Okubo algebras with isotropic norm over an arbitrary field are exactly the algebras of the form $\Okubo_{\alpha,\beta}$.
    
    \begin{theorem}[{\cite[Theorem~7]{Eld99}, \cite[Theorem~4.2]{Eld23}}] \label{theorem:isotropic-norm}
        An Okubo algebra over a field~$\F$ has isotropic norm if and only if it is isomorphic to $\Okubo_{\alpha,\beta}$ for some $\alpha, \beta \in \F \setminus \{ 0 \}$.
    \end{theorem}

    In the case when an Okubo algebra with isotropic norm contains a nonzero idempotent, its multiplication table can be simplified even more.

    \begin{theorem}[{\cite[Theorem~3.18]{Eld09}, \cite[Theorem~5.9]{Eld18}}] \label{theorem:table-of-isotropic-Okubo-with-idempotent}
		An Okubo algebra $\Okubo$ with isotropic norm over an arbitrary field~$\F$ contains a nonzero idempotent if and only if it is isomorphic to
		\begin{enumerate} [\rm (1)]
			\item $\Okubo_{1,1}$ if $\chrs \F \neq 3$;
			\item $\Okubo_{1,\beta}$ for some $\beta \in \F \setminus \{0\}$ if $\chrs \F = 3$. 
		\end{enumerate}
	\end{theorem}

    \begin{remark}
    Since the algebra $\Okubo_{1,\beta}$ has isotropic norm and contains a nonzero idempotent (e.g., one can take $z_{1,0} + z_{2,0}$), it is isomorphic to $\Okubo_{1,1}$ when $\chrs \F \neq 3$. Hence in what follows we will use the isomorphism between the algebra $\Okubo_{1,\beta}$ and an arbitrary Okubo algebra with isotropic norm and nonzero idempotents, regardless of the characteristic of the field~$\F$.
    \end{remark}

    In the case when $\chrs \F = 3$, the split Okubo algebra $\Okubo_{1,1}$ can also be characterized in terms of idempotents.

    \begin{definition}[{\cite[Definition~22]{Eld15}, \cite[Definition~6.2]{Eld18}}] \label{definition:idempotents}
    Let $\chrs \F = 3$, $\Okubo$ be an Okubo algebra over~$\F$, and $f \in \Okubo$ be an idempotent. Then $f$ is said to be:
    \begin{itemize}
        \item {\em quaternionic}, if its centralizer contains a para-quaternion algebra;
        \item {\em quadratic}, if its centralizer contains a para-quadratic algebra and no para-quater\-nion subalgebra;
        \item {\em singular}, otherwise.
    \end{itemize}
    \end{definition}

    \begin{theorem}[{\cite[Proposition~9.9 and Theorem~9.13]{CEKT13}, \cite[Lemma~21]{Eld15}, \cite[Corollary~6.5]{Eld18}}]
    Let $\chrs \F = 3$, and $\Okubo$ be an Okubo algebra over~$\F$. Then the following conditions are equivalent:
    \begin{enumerate}[{\rm (1)}]
        \item $\Okubo$ is the split Okubo algebra;
        \item the norm on~$\Okubo$ is isotropic, and $\Okubo$ contains a quaternionic idempotent;
        \item the norm on~$\Okubo$ is isotropic, and $\Okubo$ contains a singular idempotent.
    \end{enumerate}
    \end{theorem}

    The relationship between the quaternionic idempotent and all other idempotents (quadratic and singular) of the split Okubo algebra is described in~\cite[Proposition~8.7]{Eld18}.

    It is also known that if $\chrs \F = 3$, then the norm on any Okubo algebra over~$\F$ is isotropic~\cite[Theorem~5.1]{Eld97}. Besides, any Okubo algebra over a finite field is split~\cite[Theorem~3.2]{Eld23}.

    \medskip

    The following proposition utilizes the classification of Okubo algebras via central simple algebras of degree~$3$ which was obtained in~\cite{EM93}. It should be noted that these results were originally stated for $\chrs \F \notin \{ 2, 3 \}$, however, they still remain valid for $\chrs \F = 2$, see~\cite[pp.~286--287]{Eld97}.

    \begin{proposition}[{\rm \cite[p.~2486, p.~2501, Propositions~7.3 and~7.4]{EM93}, \cite[Corollary~3.6]{EP96}}] \label{proposition:omega-isotropic}
    Let $\chrs \F \neq 3$, and $\F$ contains a primitive cube root of unity~$\omega$. Consider an arbitrary Okubo algebra $\Okubo$ over~$\F$. Then 
    \begin{enumerate}[{\rm (1)}]
        \item the norm on~$\Okubo$ is isotropic;
        \item if $\Okubo$ contains a nonzero idempotent, then $\Okubo$ is isomorphic to the pseudo-octonion algebra $P_8(\F)$.
    \end{enumerate}
    \end{proposition}

    \begin{remark}
        Proposition~\ref{proposition:omega-isotropic}(2) follows immediately from Theorem~\ref{theorem:table-of-isotropic-Okubo-with-idempotent} and Propositions~\ref{proposition:split-pseudo-octonion} and~\ref{proposition:omega-isotropic}(1).
    \end{remark}

    \subsection{Zero divisors in Okubo algebras}

    According to the following well-known criterion, the existence of zero divisors in an arbitrary finite-dimensional composition algebra~$\A$ is equivalent to the norm on~$\A$ being isotropic.
	
	\begin{proposition}[{\cite[Lemma~2.1]{EM91}}] \label{proposition:norm-of-zero-divisor}
		Let $\A$ be a finite-dimensional composition algebra. A nonzero element $a \in \A$ is a zero divisor if and only if~$\n(a) = 0$.
	\end{proposition}

    In the study of orthogonality graphs and zero divisor graphs we are only interested in Okubo algebras with zero divisors, so from now onward $\Okubo$ denotes an Okubo algebra with isotropic norm over an arbitrary field~$\F$.

    We will need the following definition which introduces various subspaces that annihilate a given element~$a$ of an arbitrary algebra~$\A$.
	
	\begin{definition}
        \leavevmode
		\begin{itemize}
			\item 
			The {\em orthogonalizer} of $a$ is $O_\A(a)=\big\{ b \in \A \mid  ab=ba=0 \big\}$, i.e., the set of all elements in~$\A$ which are orthogonal to $a$.
			\item
			The {\em left annihilator} of $a$ is the set $l.\Ann_\A(a) = \big\{ b \in \A \mid ba=0 \big\}$.
			\item
			The {\em right annihilator} of $a$ is the set $r.\Ann_\A(a) = \big\{ b \in \A \mid ab=0 \big\}$.
		\end{itemize}
	\end{definition}

    \medskip

    We will use the explicit form of left and right annihilators, as well as their intersections, of zero divisors in~$\Okubo$ which was obtained in~\cite{Matzri}.

    \begin{notation}
    Let $x \in \Okubo$. The set of elements which are orthogonal to the element~$x$ with respect to the symmetric bilinear form $\n(\cdot,\cdot)$ is denoted by \[ x^\perp = \{ a \in \Okubo \mid \n(x,a) = 0 \}. \]
    \end{notation}

    \begin{lemma}[{\cite[Theorem~3.1, Proposition~3.4]{Matzri}}] \label{lemma:annihilator-formula}
    Let $x \in Z(\Okubo)$. Then the left and right annihilators of the element~$x$ have dimension~$4$ and are given by the formulae
    \begin{align*}
        l.\Ann_{\Okubo}(x) &= x * \Okubo = \{ x * a \mid a \in \Okubo \},\\
        r.\Ann_{\Okubo}(x) &= \Okubo * x = \{ a * x \mid a \in \Okubo \}.
    \end{align*}
    \end{lemma}

    \begin{lemma}[{\cite[Proposition~3.8]{Matzri}}] \label{lemma:annihilator-intersection}
    Let $x, y \in Z(\Okubo)$.
    \begin{enumerate}[{\rm (1)}]
        \item If $x * y \neq 0$, then $(x * \Okubo) \cap (\Okubo * y) = \F (x * y)$ has dimension~$1$.
        \item If $x * y = 0$, then $(x * \Okubo) \cap (\Okubo * y) = x * y^{\perp} = x^{\perp} * y$ has dimension~$3$.
    \end{enumerate}
    \end{lemma}

    Note that {\cite[Proposition~3.7]{Matzri}} also describes the intersection of two left (or two right) annihilators of linearly independent zero divisors in~$\Okubo$, but it will not be used in the current paper.

    Lemma~\ref{lemma:annihilator-intersection} immediately implies the following result on the directed zero divisor graph of the algebra~$\Okubo$. 

    \begin{theorem} \label{theorem:graph-of-zero-divisors-main}
		Let $\Okubo$ be an Okubo algebra with isotropic norm over a field~$\F$. Then the graph $\Gamma_Z(\Okubo)$ is strongly connected, and $\diam\Gamma_Z(\Okubo) = 2$.
	\end{theorem}
	
	\begin{proof}
		Let $x, y \in Z(\Okubo)$. By Lemmas~\ref{lemma:annihilator-formula} and~\ref{lemma:annihilator-intersection}, the set $r.\Ann(x) \cap l.\Ann(y) = (\Okubo * x) \cap (y * \Okubo)$ contains a nonzero element~$z$, so $[x] \rightarrow [z] \rightarrow [y]$. Hence $\Gamma_Z(\Okubo)$ is connected, and $\diam\Gamma_Z(\Okubo) \leq 2$.
        
		On the other hand, if the diameter were equal to~$1$, then for any pair of zero divisors $x, y \in \Okubo$ the identity $x*y = 0$ would hold, which contradicts the multiplication table~\ref{table:okubo-algebra-isotropic}.
	\end{proof}

    Another consequence of Lemma~\ref{lemma:annihilator-intersection} is the explicit form of the orthogonalizer of an arbitrary zero divisor in~$\Okubo$.
    
    \begin{corollary} \label{corollary:orthogonalizer-formula}
    Let $x \in Z(\Okubo)$.
    \begin{enumerate}[{\rm (1)}]
        \item If $x * x \neq 0$, then $O_{\Okubo}(x) = \F (x * x)$ has dimension~$1$.
        \item If $x * x = 0$, then $O_{\Okubo}(x) = x * x^{\perp} = x^{\perp} * x$ has dimension~$3$.
    \end{enumerate}
    \end{corollary}

    \begin{proof}
    Can be obtained by substituting $y = x$ in Lemma~\ref{lemma:annihilator-intersection}.
    \end{proof}

    \section{Orthogonality graph of an Okubo algebra}

    In this section we describe the connected components of the orthogonality graph of an Okubo algebra~$\Okubo$ with isotropic norm over an arbitrary field~$\F$ and compute their diameters. Meanwhile, zero divisors are divided naturally into two types, depending on whether the value $\n(x, x*x)$ is nonzero.

    We first show that every zero divisor $x \in \Okubo$ which satisfies $\n(x, x*x) \neq 0$ is contained in a two-element connected component of the orthogonality graph $\Gamma_O(\Okubo)$.

    \begin{proposition} \label{proposition:connectivity-component-with-type3-elem}
        Let $x \in Z(\Okubo)$ satisfy $\n(x, x*x) \neq 0$.
        Then the subgraph of $\Gamma_O(\Okubo)$ on the vertex set
        \[
        V_x = \{ [x], [x*x] \}
        \]
        is a connected component of diameter~$1$.
    \end{proposition}

    \begin{proof}
        The condition $\n(x, x*x) \neq 0$ guarantees that $x*x \neq 0$. We show that the vertices $[x]$ and $[x*x]$ are distinct. Indeed, if $[x] = [x*x]$, then $x*x = \gamma x$ for some $\gamma \in \F$, and thus $\n(x,x*x) = \n(x,\gamma x) = 2 \gamma \n(x) = 0$, a contradiction.
        
        By Corollary~\ref{corollary:orthogonalizer-formula}(1), we have $O_{\Okubo}(x) = \F (x*x)$. 
        It follows from Eq.~\eqref{equation:xxxx} that
        \[
        (x*x)*(x*x) = \n(x, x*x) x - \n(x)x*x = \n(x, x*x) x \neq 0.
        \]
        Consequently, the element $x*x$ also satisfies the conditions of Corollary~\ref{corollary:orthogonalizer-formula}(1), so
        \[O_{\Okubo}(x*x) = \F (x*x)*(x*x) = \F x. \] 
        Then one can easily see that $V_x$ is a connected component of diameter~$1$.
    \end{proof}

    Next, we study the connected components of~$\Gamma_O(\Okubo)$ formed by all other zero divisors. First, note that if an Okubo algebra~$\Okubo$ contains a nonzero element $x$ which satisfies the condition $n(x, x*x) = 0$, then $\Okubo$ must contain a nonzero idempotent. If, moreover, the norm on~$\Okubo$ is isotropic, then, by Theorem~\ref{theorem:table-of-isotropic-Okubo-with-idempotent}, $\Okubo$ is isomorphic to $\Okubo_{1,\beta}$ for some $0 \neq \beta \in \F$.
	
	\begin{proposition}[{\cite[Proposition~3.3]{Eld97}, \cite[(34.10)]{KMRT98}}] \label{proposition:idempotent-existence}
		Let $(\mathcal{S}, *, \n)$ be a symmetric composition algebra over a field~$\F$. Assume that there exists a nonzero element $x \in \mathcal{S}$ such that $\n(x, x*x) = 0$. Then $\mathcal{S}$ contains a nonzero idempotent.
	\end{proposition}

    Let $x \in Z(\Okubo)$ be such that $\n(x, x*x) = 0$. We have either $x*x = 0$, and then, by Corollary~\ref{corollary:orthogonalizer-formula}, the orthogonalizer of~$x$ has dimension~$3$, or $x*x \neq 0$, so $O_{\Okubo}(x) = \F (x * x)$, and $[x] \leftrightarrow [x * x]$ is the unique edge in $\Gamma_O(\Okubo)$ incident to~$[x]$. In the second case, Eq.~\eqref{equation:xxxx} implies that $(x*x)*(x*x) = 0$. 
    
    We begin by studying paths between elements of the first of these two types.

    \begin{lemma} \label{lemma:length-two-path}
        Let $x, y \in Z(\Okubo)$ be such that
        \begin{enumerate}[{\rm (1)}]
            \item $\D([x],[y]) > 1$, i.e., $x, y$ are linearly independent and are not orthogonal to each other;
            \item $x*x = y*y = 0$.
        \end{enumerate}
        Then $\D([x],[y]) = 2$ if and only if $\n(x,y) = 0$. In this case, there exists a unique path of length~$2$ between the vertices $[x]$ and $[y]$ which has the form $[x] \leftrightarrow [y*x] \leftrightarrow [y]$ (for $y*x \neq 0$) or $[x] \leftrightarrow [x*y] \leftrightarrow [y]$ (for $x*y \neq 0$).
    \end{lemma}

    \begin{proof}
    Since the graph $\Gamma_O(\Okubo)$ is undirected, and the elements $x, y$ are not orthogonal to each other, we may assume without loss of generality that $y * x \neq 0$.
    
    Note that if $\n(x,y) = 0$, then $[x] \leftrightarrow [y * x] \leftrightarrow [y]$. One can either verify this directly by using Theorem~\ref{theorem:symmetric-flexible}(2) and Eq.~\eqref{equation:lin-of-symm} or apply Corollary~\ref{corollary:orthogonalizer-formula}(2) which implies $y * x \in O_{\Okubo}(x) \cap O_{\Okubo}(y)$.

    Now assume that there exists $z \neq 0$ such that $[x] \leftrightarrow [z] \leftrightarrow [y]$. Then, by Lemma~\ref{lemma:annihilator-formula}, we have $z \in (y * \Okubo) \cap (\Okubo * x)$, and it follows from Lemma~\ref{lemma:annihilator-intersection}(1) that $z = y * x$ up to proportionality. Since $[z] \leftrightarrow [y]$, it holds that $y * (y * x) = y * z = 0$. By applying Eq.~\eqref{equation:lin-of-symm} we obtain that
    \[
    \n(x,y)y = x * (y * y) + y * (y * x) = x * 0 = 0,
    \]
    and thus $\n(x, y) = 0$. Note that the desired equality can be obtained directly from {\cite[Proposition~3.7]{Matzri}} by using the fact that $z \in (x * \Okubo) \cap (y * \Okubo)$.
    \end{proof}

    \begin{proposition} \label{proposition:middle-element}
    Under the conditions of Lemma~\ref{lemma:length-two-path} it holds that
    \[
    (x*y)*(x*y) = (y*x)*(y*x) = 0. 
    \]
    \end{proposition}

    \begin{proof}
    It is sufficient to prove this statement for $y*x$. If $y*x = 0$, then it is obvious. Now assume that $y*x \neq 0$. Then we can either verify the desired equality directly or use Corollary~\ref{corollary:orthogonalizer-formula}: since $x, y \in O_{\Okubo}(y*x)$ are linearly independent, we have $\dim O_{\Okubo}(y*x) \geq 2$.
    \end{proof}

    \begin{corollary} \label{corollary:distance-at-most-two}
    Let $x, y \in Z(\Okubo)$ be such that $x*x = y*y = 0$. Then $\D([x],[y]) \leq 2$ if and only if $\n(x,y) = 0$.
    \end{corollary}

    \begin{proof}
    In view of Lemma~\ref{lemma:length-two-path}, it is sufficient to show that if $\D([x],[y]) \leq 1$, then $\n(x,y) = 0$. If $\D([x],[y]) = 0$, i.e., $[x] = [y]$, then $y = \gamma x$ for some $\gamma \in \F$, so $\n(x,y) = \gamma\n(x,x) = 2\gamma \n(x) = 0$. If $\D([x],[y]) = 1$, then $y \in O_{\Okubo}(x)$ and, by Corollary~\ref{corollary:orthogonalizer-formula}(2) and Eq.~\eqref{equation:symmetric},
    \[
    \n(x, O_\Okubo(x)) = \n(x, x*x^\perp) = \n(x*x, x^\perp) = 0. \qedhere
    \]
    \end{proof}

    \begin{lemma} \label{lemma:zero-square-subspace}
    Assume that there exists a pair of linearly independent elements $a, b \in \Okubo$ with
    \[
    a*a=a*b=b*a=b*b=0.
    \]
    Let $x \in Z(\Okubo)$ be such that $x*x = 0$. Then there exists $y \in O_{\Okubo}(x)$, $y \notin \F x$, which satisfies $y*y = 0$. Moreover, for any $z \in \spn \{ x, y \} \subset O_{\Okubo}(x)$ we have $z*z = 0$.
    \end{lemma}

    \begin{proof}
    Consider a two-dimensional subspace $W = \spn \{ a, b \}$. Note that for any $w \in W$ we have $w*w = 0$. Since the subspace $x^{\perp}$ has dimension~$7$, there exists a nonzero element $w \in x^{\perp} \cap W$. Assume first that $w \notin \F x$. If $w \in O_{\Okubo}(x)$, then we set $y = w$. Otherwise, we apply Lemma~\ref{lemma:length-two-path} and Proposition~\ref{proposition:middle-element} to the pair of elements $x$ and $w$ and take a nonzero element $y \in \{ x*w, w*x \}$.

    Consider now the case when $w \in \F x$. Then $x \in W$ implies that $a, b \in O_{\Okubo}(x)$, and we choose $y \in \{ a, b \}$ so that $y \notin \F x$.

    Finally, it follows from $x*x = x*y = y*x = y*y = 0$ that for any $z \in \spn \{ x, y \} \subset O_{\Okubo}(x)$ we have $z*z = 0$.
    \end{proof}

    \begin{example} \label{example:length-two-path}
    We use Table~\ref{table:okubo-algebra-isotropic} with $\alpha = 1$ to define $x = z_{0, 2} + z_{1, 2} + z_{2, 2}$. One can verify that $x*x = 0$. Consider $a = z_{0,1} - z_{1,1}$ and $b = z_{0,1} - z_{2,1}$. Then $x = a*a = b*b$ and $\n(a) = \n(b) = 0$, hence $a, b \in O_{\Okubo}(x)$. By Corollary~\ref{corollary:orthogonalizer-formula}(2), we have $\dim O_{\Okubo}(x) = 3$, so $O_{\Okubo}(x) = \spn \{ x, a, b \}$.

    We now find out when there exists an element $y \in O_{\Okubo}(x)$ such that $y \notin \F x$ and $y * y = 0$. Clearly, it is sufficient to consider only $y \in \spn \{ a, b\}$, $y \neq 0$, i.e., $y = \gamma a + \delta b$ for some $\gamma, \delta \in \F$ which are not both zero. Then
    \[
    0 = y*y = ((\gamma + \delta) z_{0,1} - \gamma z_{1,1} - \delta z_{2,1})^2 = (\gamma^2 + \gamma \delta + \delta^2) (z_{0,2} + z_{1,2} + z_{2,2})
    \]
    is equivalent to $\gamma^2 + \gamma \delta + \delta^2 = 0$.
    
    If $\chrs \F = 3$, then $\gamma = \delta$, and for $\gamma = \delta = -1$ we obtain $y = z_{0,1} + z_{1,1} + z_{2,1}$ which is the unique solution up to proportionality.
    
    If $\chrs \F \neq 3$, then the existence of a nonzero solution is equivalent to the existence of a primitive cube root of unity $\omega \in \F$. The only two solutions up to proportionality are $y = \omega^2 z_{0,1} + \omega z_{2,1} + z_{1,1}$ (for $\gamma = -1$ and $\delta = -\omega$) and $y = \omega^2 z_{0,1} + \omega z_{1,1} + z_{2,1}$ (for $\gamma = -\omega$ and $\delta = -1$). Let us dwell on the first case. Define
    \[
    z = \omega^2 z_{0,2} + \omega z_{1,2} + z_{2,2} = (\omega z_{0,1} - z_{1,1})^2 = \omega (\omega z_{1,1} - z_{2,1})^2.
    \]
    Note that $z*z = 0$, $z \notin \spn \{ x, a, b \} = O_{\Okubo}(x)$ and
    \[y \in \spn \{ \omega z_{0,1} - z_{1,1}, \omega z_{1,1} - z_{2,1} \} \subseteq O_{\Okubo}(z).\]
    Therefore, $[x] \centernot\leftrightarrow [z]$ and $[x] \leftrightarrow [y] \leftrightarrow [z]$. Hence the pair of elements $x$ and $z$ satisfies the conditions of Lemma~\ref{lemma:length-two-path}.
    \end{example}

    \begin{corollary} \label{corollary:field-condition}
    Let $x \in \Okubo$ and $x*x = 0$. Then there exists an element $y \in O_{\Okubo}(x)$ such that $y \notin \F x$ and $y*y = 0$ if and only if one of the following conditions holds:
    \begin{enumerate}[{\rm (1)}]
        \item $\chrs \F = 3$;
        \item $\chrs \F \neq 3$, and $\F$ contains a primitive cube root of unity $\omega$.
    \end{enumerate}
    \end{corollary}

    \begin{proof}
    By Lemma~\ref{lemma:zero-square-subspace}, the desired condition is true or false for all $x \in \Okubo$ which satisfy $x*x = 0$ simultaneously. According to Example~\ref{example:length-two-path}, the fulfillment of this condition for $x = z_{0, 2} + z_{1, 2} + z_{2, 2}$ is equivalent to the fact that either $\chrs \F = 3$, or $\chrs \F \neq 3$ and $\omega \in \F$.
    \end{proof}

    \begin{lemma} \label{lemma:star-graph}
    Let $\chrs \F \neq 3$ and $\omega \notin \F$. For any $x \in Z(\Okubo)$ such that $x*x = 0$ the set
	\[ V'_{x} = \big\{ [y] = \F y \mid y \in O_{\Okubo}(x) \big\} \]
    forms a connected component in $\Gamma_O(\Okubo)$. This connected component is a star graph with $[x]$ being its internal node, and its diameter equals~$2$.
    \end{lemma}

    \begin{proof}
    By Corollary~\ref{corollary:field-condition}, for any $y \in O_{\Okubo}(x)$, $y \notin \F x$, we have $y*y \neq 0$. Then, by Corollary~\ref{corollary:orthogonalizer-formula}(1), $O_{\Okubo}(y) = \F(y*y) = \F x$. Hence the subgraph of $\Gamma_O(\Okubo)$ on the vertex set $V'_{x}$ is a connected component, and it is a star whose internal node is $[x]$. It follows from Corollary~\ref{corollary:orthogonalizer-formula}(2) that $\dim O_{\Okubo}(x) = 3$, so this star has more than one leaf. Therefore, its diameter equals~$2$.
    \end{proof}

    \begin{lemma} \label{lemma:diameter-five}
    Assume that either $\chrs \F = 3$, or $\chrs \F \neq 3$ and $\omega \in \F$. Then the set
    \[ V = \big\{ [x] = \F x \mid x \in Z(\Okubo), \: \n(x, x*x) = 0 \big\} \]
    forms a connected component in $\Gamma_O(\Okubo)$ whose diameter is at most~$5$.
    \end{lemma}

    \begin{proof}
    Consider an arbitrary element $x \in Z(\Okubo)$ such that $\n(x, x*x) = 0$. If $x*x \neq 0$, then, by Corollary~\ref{corollary:orthogonalizer-formula}(1), $O_{\Okubo}(x) = \F (x * x)$, so $[x] \leftrightarrow [x * x]$ is the unique edge incident to $[x]$. In this case Eq.~\eqref{equation:xxxx} implies that $(x*x)*(x*x) = 0$.

    Therefore, in order to prove the statement, it is sufficient to show that there exists a path of length at most~$3$ between any two nonzero elements $x$ and $y$ such that $x * x = y * y = 0$. By Corollary~\ref{corollary:distance-at-most-two}, $\D([x],[y]) \leq 2$ if and only if $\n(x,y) = 0$.
    
    Assume that $\n(x,y) \neq 0$. By Lemma~\ref{lemma:zero-square-subspace} and Corollary~\ref{corollary:field-condition}, there exists a two-dimensional subspace $W \subset O_{\Okubo}(y)$ such that $w*w = 0$ for all $w \in W$. Note that $y \in W$ and $\n(x,y) \neq 0$, so $\dim (W \cap x^{\perp}) = 1$. Choose a nonzero element $z \in \Okubo$ such that $W \cap x^{\perp} = [z]$. By Lemma~\ref{lemma:length-two-path}, we have $\D([x],[z]) = 2$, so $\D([x],[y]) = 3$.
    \end{proof}

    \begin{proposition} \label{proposition:m-paths-length-three}
    Let $x, y \in Z(\Okubo)$ be such that $x * x = y * y = 0$ and $\D([x],[y]) = 3$, that is, $\n(x,y) \neq 0$. If all the elements $w \in O_{\Okubo}(y)$ which satisfy $w*w = 0$ constitute the union of~$m$ distinct two-dimensional subspaces $W \ni y$, then there exist exactly~$m$ paths of length~$3$ between $[x]$ and $[y]$.
    \end{proposition}

    \begin{proof}
    Follows from the proof of Lemma~\ref{lemma:diameter-five}, since $W \cap x^{\perp} = [z]$ implies $W = \spn \{ y, z \}$. Therefore, distinct subspaces $W$ correspond to distinct lines $[z]$. Besides, by Lemma~\ref{lemma:length-two-path}, there exists a unique path of length~$2$ between $[x]$ and $[z]$.
    \end{proof}

    \begin{remark} \label{remark:max-length-of-path}
    Let $x, y \in Z(\Okubo)$ be such that
    \begin{enumerate}[{\rm (1)}]
        \item $n(x, x*x) = n(y, y*y) = 0$;
        \item $x * x \neq 0 \neq y * y$;
        \item $n(x*x, y*y) \neq 0$.
    \end{enumerate}
    Then it follows from the proof of Lemma~\ref{lemma:diameter-five} that $\D([x],[y]) = 5$.
    \end{remark}

    \begin{example} \label{example:distance-five}
    Let $\chrs \F \neq 3$ and $\omega \in \F$. We use Table~\ref{table:okubo-algebra-isotropic} with $\alpha = 1$ to define $x = z_{0,1} - z_{1, 1}$ and $y = z_{0, 2} - z_{2, 2}$. Then 
	\[
    x*x = z_{0, 2} + z_{1, 2} + z_{2, 2}, \quad y*y = \beta (z_{0, 1} + z_{2, 1} + z_{1, 1}).
    \]
    Applying Eqs.~\eqref{equation:norm-of-basis} and~\eqref{equation:scpr-of-basis}, we obtain:
    \begin{align*}
        &\n(x) = \n(z_{0,1}) + \n(z_{1, 1}) - \n(z_{0,1}, z_{1, 1}) = 0, \\
        &\n(y) = \n(z_{0,2}) + \n(z_{2, 2}) - \n(z_{0,2}, z_{2, 2})= 0, \\
        &\n(x, x*x) = \n(z_{0, 1}, z_{0, 2}) - \n(z_{1, 1}, z_{2, 2}) = \beta - \beta = 0, \\
        &\n(y, y*y) = \beta (\n(z_{0, 2}, z_{0, 1}) - \n(z_{2, 2}, z_{1, 1})) = \beta^2 - \beta^2 = 0, \\
        &\n(x*x, y*y) = \beta (\n(z_{0,2}, z_{0, 1}) + \n(z_{1,2}, z_{2, 1}) + \n(z_{2,2}, z_{1, 1}))= 3\beta^2 \neq 0.        
    \end{align*}
    Therefore, by Remark~\ref{remark:max-length-of-path}, we have $\D([x],[y]) = 5$, and $\diam V = 5$.
    \end{example}

    \begin{lemma} \label{lemma:orthogonality-characteristic-three}
    Let $\chrs \F = 3$. Then the diameter of the connected component of $\Gamma_O(\Okubo)$ on the vertex set
    \[ V = \big\{ [x] = \F x \mid x \in Z(\Okubo), \: \n(x, x*x) = 0 \big\} \]
    equals~$5$ if $\Okubo$ is the split Okubo algebra, and $3$ otherwise.
    \end{lemma}

    \begin{proof}
    By Lemma~\ref{lemma:diameter-five}, $\diam V \leq 5$.
    Assume first that $\Okubo$ is the split Okubo algebra, so its multiplication table is Table~\ref{table:okubo-algebra-isotropic} with $\alpha = \beta = 1$. We define $x = z_{0,1} - z_{1, 1}$ and $y = z_{1,0} - z_{2, 1}$. Then 
	\[
    x*x = z_{0, 2} + z_{1, 2} + z_{2, 2}, \quad y*y = z_{2, 0} + z_{0,1} + z_{1, 2}.
    \]
    As we have computed earlier, $\n(x) = \n(x,x*x) = 0$. Moreover, applying Eqs.~\eqref{equation:norm-of-basis} and~\eqref{equation:scpr-of-basis}, we obtain:
    \begin{align*}
        &\n(y) = \n(z_{1,0}) + \n(z_{2, 1}) - \n(z_{1,0}, z_{2, 1}) = 0, \\
        &\n(y, y*y) = \n(z_{1, 0}, z_{2, 0}) - \n(z_{2, 1}, z_{1, 2}) = 0, \\
        &\n(x*x, y*y) = \n(z_{0,2}, z_{0, 1})= 1 \neq 0.        
    \end{align*}
    Therefore, by Remark~\ref{remark:max-length-of-path}, we have $\D([x],[y]) = 5$, and $\diam V = 5$.

    Assume now that $\Okubo$ is not split. In particular, this means that $\Okubo$ is isomorphic to the algebra $\Okubo_{1,\beta}$, and $\beta \neq k^3$ for any $k \in \F$. Indeed, otherwise we could consider an isomorphism $\Okubo_{1,\beta} \to \Okubo_{1,1}$ such that $z_{i,j} \mapsto k^j z_{i,j}$ where $0 \leq i, j \leq 2$ and $(i, j) \neq (0,0)$. 
    
    Consider the elements $x$ and $y$ defined in Example~\ref{example:distance-five}. Then $x*x$ and $y*y$ are linearly independent and orthogonal to each other (see also Example~\ref{example:length-two-path}), and thus $\D([x],[y]) = 3$. Consequently, $\diam V \geq 3$. Since $V$ is a connected component in $\Gamma_O(\Okubo)$ and $[x] \in V$, the elements of $V$ are exactly those vertices which are connected to $[x]$ by a path. By Corollary~\ref{corollary:orthogonalizer-formula}(2), we have $\dim O_{\Okubo}(x*x) = \dim O_{\Okubo}(y*y) = 3$, so $O_{\Okubo}(x*x) = \spn \{ x, x*x, y*y \}$ and $O_{\Okubo}(y*y) = \spn \{ y, x*x, y*y \}$. For any $0 \neq k \in \F$ we set 
    \[
    u_k = x*x + k/\beta \cdot y*y = (z_{0, 2} + z_{1, 2} + z_{2, 2}) + k (z_{0, 1} + z_{2, 1} + z_{1, 1}).
    \]
    Note that $v_k = k z_{0,1} - z_{0,2} \in u_k^{\perp}$, so, by Corollary~\ref{corollary:orthogonalizer-formula}(2), 
    \[
    w_k = v_k * u_k = \beta k z_{1,0} - \beta k z_{2,0} - \beta z_{0,1} + k^2 z_{0,2} + \beta z_{1,1} - k^2 z_{2,2} \in O_{\Okubo}(u_k).
    \]
    One can easily infer that $O_{\Okubo}(u_k) = \spn \{ x*x, y*y, w_k \}$. We now find out when $w_k * w_k = 0$ holds. By using the fact that the norm on~$\Okubo$ is symmetric, Eq.~\eqref{equation:lin-of-symm} and the equality $u_k * u_k = 0$, we obtain
    \begin{align*}
    w_k * w_k &= (v_k * u_k) * (v_k * u_k) = \n(v_k * u_k, u_k) v_k - u_k * (v_k * (v_k * u_k)) \\
    &= \n(v_k, u_k * u_k) v_k - u_k * (\n(v_k, u_k) v_k - u_k * (v_k * v_k)) = u_k * ( u_k * (v_k * v_k)) \\
    &= \n(u_k, v_k * v_k) u_k - (v_k * v_k) * (u_k * u_k) = \n(u_k, v_k * v_k) u_k.
    \end{align*}
    Since $v_k*v_k = \beta z_{0,1} + k^2 z_{0,2}$, it holds that 
    \[
    \n(u_k, v_k * v_k) = \beta \n(z_{0,2}, z_{0,1}) + k^3 \n(z_{0,1}, z_{0,2}) = \beta (\beta + k^3).
    \]
    Hence the condition $w_k * w_k = 0$ is equivalent to $\beta = (-k)^3$. Therefore, according to our assumption on the parameter $\beta$, we have $w_k * w_k \neq 0$. 
    
    We set $v'_k = z_{2,2} - k z_{1,1} \in u_k^{\perp}$. One can verify that $w_k = u_k * v'_k$. Then Eq.~\eqref{equation:lin-of-symm} and the fact that $u_k \in O_{\Okubo}(x*x)$ imply
    \begin{align*}
        (x*x)*w_k + w_k*(x*x) &= (x*x)*(u_k*v'_k) + (v_k*u_k)*(x*x) \\
        &= \n(x*x,v'_k) u_k - v'_k*(u_k*(x*x)) \\
        &+ \n(x*x,v_k) u_k - ((x*x)*u_k)*v_k \\
        &= (\n(z_{2,2},-k z_{1,1}) + \n(z_{0,2},k z_{0,1})) u_k = 0.
    \end{align*}
    Similarly, $(y*y)*w_k + w_k*(y*y) = 0$. It then follows from $z \in O_{\Okubo}(u_k) = \spn \{ x*x, y*y, w_k \}$ and $z*z = 0$ that $z \in \spn \{ x*x, y*y \}$.

    By Corollary~\ref{corollary:orthogonalizer-formula}(1), any path which starts at the vertex $[x]$ must pass through the vertex $[x*x]$. Next, $[x*x]$ is adjacent to vertices $[z] \in \mathbb{P}(\spn \{x, x*x, y*y))$. If the coefficient at $x$ in $z$ is nonzero, then it follows from Example~\ref{example:length-two-path} that $z*z \neq 0$, and the vertex $[z]$ is adjacent only to $[x*x]$. Otherwise, $z \in \spn\{ x*x, y*y \}$, so either $[z] = [y*y]$, or $[z] = [u_k]$ for some $0 \neq k \in \F$. In both cases $[z]$ is connected by an edge either to vertices $[x*x]$, $[y*y]$, $[u_k]$ again, or to vertices of the form $[q]$, where $q \in O_{\Okubo}(z)$ and $q*q \neq 0$, and thus $[q]$ is adjacent only to $[z]$. Note that the proof of the previous statement for $[z] = [y*y]$ is similar to Example~\ref{example:length-two-path}. Therefore, $\diam V = 3$.
    \end{proof}

    \begin{theorem} \label{theorem:graph-of-orthogonality}
    Let $\Okubo$ be an Okubo algebra with isotropic norm over an arbitrary field~$\F$. Then the orthogonality graph $\Gamma_O(\Okubo)$ is disconnected, and the vertex sets of its connected components are as follows:
		\begin{enumerate} [\rm (1)]
			\item The set
			\[ V_{x} = \big\{ [x], \, [x*x] \big\} \]
            for any $x \in Z(\Okubo)$ such that $\n(x, x*x) \neq 0$;
                \item If $\Okubo$ contains a nonzero idempotent, $\chrs \F \neq 3$ and $\omega \notin \F$, then we consider the set
			\[ V'_{x} = \big\{ [y] \mid y \in O_{\Okubo}(x) \big\} \]
            for any $x \in Z(\Okubo)$ such that $x*x = 0$;
                \item If $\Okubo$ contains a nonzero idempotent but $\chrs \F = 3$ or $\omega \in \F$, then we consider the set
			\[ V = \big\{ [x] \mid x \in Z(\Okubo), \: \n(x, x*x) = 0 \big\}. \]
		\end{enumerate}
		The diameter of the connected component on the vertex set $V_x$ equals $1$, on the vertex set $V'_x$ --- equals~$2$, and on the vertex set $V$ --- equals $5$ if $\Okubo$ is the split Okubo algebra, and $3$ otherwise.
        
        Note that, by Theorem~\ref{theorem:table-of-isotropic-Okubo-with-idempotent}(1), if $\Okubo$ contains a nonzero idempotent and $\chrs \F \neq 3$, then $\Okubo$ is automatically split.
    \end{theorem}

    \begin{proof}
        Follows by combining the results of Propositions~\ref{proposition:connectivity-component-with-type3-elem} and~\ref{proposition:idempotent-existence}, Lemmas~\ref{lemma:star-graph}, \ref{lemma:diameter-five}, and \ref{lemma:orthogonality-characteristic-three}, and Example~\ref{example:distance-five}.

        We remark that any Okubo algebra with isotropic norm and nonzero idempotents contains an element $x \in Z(\Okubo)$ such that $\n(x, x*x) \neq 0$, and thus there are connected components of the form $V_x$ in $\Gamma_O(\Okubo)$. Indeed, for $x = z_{1,0}$ we have $x*x=z_{2,0}$ and $\n(x,x*x) = 1$.
    \end{proof}

    \section{Number of shortest paths}

    In this section we study the number of shortest paths between an arbitrary pair of vertices in the orthogonality graph of an Okubo algebra with isotropic norm. First, we note that if $\Okubo$ is not a split Okubo algebra over a field~$\F$, such that either $\chrs \F = 3$ or $\F$ contains a primitive cube root of unity~$\omega$, then all shortest paths in $\Gamma_O(\Okubo)$ are unique.

    \begin{proposition} \label{proposition:unique-path}
    Let $\Okubo$ be an Okubo algebra with isotropic norm over a field~$\F$. Assume that at least one of the following conditions is satisfied:
    \begin{enumerate}[{\rm (1)}]
        \item $\Okubo$ does not contain nonzero idempotents;
        \item $\chrs \F \neq 3$ and $\omega \notin \F$;
        \item $\chrs \F = 3$, and the algebra $\Okubo$ is not split.
    \end{enumerate}
    Then all shortest paths in $\Gamma_O(\Okubo)$ are unique. In other words, if $\D([x],[y]) = k < \infty$ for some $x, y \in Z(\Okubo)$, then there exists a unique path of length~$k$ between the vertices $[x]$ and~$[y]$.
    \end{proposition}

    \begin{proof}
    The uniqueness of the shortest paths in cases (1) and (2) follows automatically from the explicit form of the connected components, see Proposition~\ref{proposition:connectivity-component-with-type3-elem} and Lemma~\ref{lemma:star-graph}, while case~(3) follows from the proof of Lemma~\ref{lemma:orthogonality-characteristic-three}.
    \end{proof}

    If an Okubo algebra $\Okubo$ with isotropic norm contains nonzero idempotents, $\chrs \F \neq 3$ and $\omega \in \F$, then, by Theorem~\ref{theorem:table-of-isotropic-Okubo-with-idempotent}(1),
    $\Okubo$ is split, and Proposition~\ref{proposition:split-pseudo-octonion} implies that it is isomorphic to the pseudo-octonion algebra $P_8(\F)$ which was introduced in Subsection~2.1. By using the matrix form of the elements in $P_8(\F)$, we show that all elements $z \in Z(P_8(\F))$ such that $z*z = 0$ belong to the same orbit under the action of the automorphism group. This observation allows to investigate the number of shortest paths between an arbitrary pair of zero divisors lying in the same connected component.

    \begin{proposition} \label{proposition:two-paths-length-three}
    Let $\Okubo$ be the split Okubo algebra over a field~$\F$, $\chrs \F \neq 3$ and $\omega \in \F$. Let the elements $x, y \in Z(\Okubo)$ be such that $x * x = y * y = 0$ and $\D([x],[y]) = 3$. Then there exist exactly two paths of length~$3$ between the vertices $[x]$ and $[y]$.
    \end{proposition}

    \begin{proof}
    In view of Proposition~\ref{proposition:m-paths-length-three}, it is sufficient to show that all elements $w \in O_{\Okubo}(y)$ such that $w*w = 0$ constitute the union of two distinct two-dimensional subspaces $W \ni y$.
    
    We use the isomorphism between $\Okubo$ and $P_8(\F)$ and consider an arbitrary element $z \in Z(P_8(\F))$ such that $z*z=0$. Then $\n(z) = 0$ and, by Eq.~\eqref{equation:product}, $z*z=zz = 0$, so the Jordan normal form of the matrix~$z$ is
    \[
    J_z = \begin{pmatrix}
			0 &  1 & 0 \\
			0  & 0 & 0 \\
			0  &  0 & 0
		\end{pmatrix}.
    \]
    Since $J_z = gzg^{-1}$ for some matrix $g \in GL_3(\F)$, the automorphism of the algebra $P_8(\F)$ defined by $a \mapsto gag^{-1}$ maps $z$ to $J_z$. It follows that all nonzero elements $z \in P_8(\F)$ which satisfy $z*z = 0$ belong to the same orbit under the action of the automorphism group on $P_8(\F)$. Example~\ref{example:length-two-path} demonstrates that all elements $w \in O_{\Okubo}(z_{0, 2} + z_{1, 2} + z_{2, 2})$ such that $w*w = 0$ constitute the union of subspaces
    \[
    \spn\{z_{0, 2} + z_{1, 2} + z_{2, 2}, \omega^2 z_{0,1} + \omega z_{2,1} + z_{1,1} \} \cup \spn\{z_{0, 2} + z_{1, 2} + z_{2, 2}, \omega^2 z_{0,1} + \omega z_{1,1} + z_{2,1} \}.
    \]
    Therefore, the desired statement is satisfied for an arbitrary $z$, including $z = y$.
    \end{proof}

    \begin{corollary} \label{corollary:unique-path}
    Let $\Okubo$ be the split Okubo algebra over a field~$\F$, $\chrs \F \neq 3$ and $\omega \in \F$. Let the elements $x, y \in Z(\Okubo)$ be such that $\D([x],[y]) = k$ in $\Gamma_O(\Okubo)$, $0 \leq k \leq 5$. Then either there exists a unique path of length $k$ between $[x]$ and $[y]$, or there are exactly two such paths. The second possibility is achieved precisely in the following cases:
    \begin{enumerate}[{\rm (1)}]
        \item $\D([x],[y]) = 3$, and $x*x = y*y = 0$;
        \item $\D([x],[y]) = 4$, and $x*x = 0$ or $y*y = 0$;
        \item $\D([x],[y]) = 5$.
    \end{enumerate}
    \end{corollary}

    \begin{proof}
    Follows immediately from Proposition~\ref{proposition:connectivity-component-with-type3-elem}, Lemma~\ref{lemma:length-two-path}, the proof of Lemma~\ref{lemma:diameter-five} and Proposition~\ref{proposition:two-paths-length-three}.
    \end{proof}

    It remains to consider the case of the split Okubo algebra over a field~$\F$ with $\chrs \F = 3$. To do this, we will need the construction of such an algebra via a Petersson algebra associated to an eight-dimensional split Hurwitz algebra over~$\F$, i.e., a unital composition algebra over~$\F$ with isotropic norm. Such a Hurwitz algebra is unique, and it is isomorphic to the Zorn vector-matrix algebra. We rely on the construction of the Petersson algebra and its properties given in Elduque's paper~\cite{Eld18}.

    \begin{definition}
    The Zorn vector-matrix algebra is the composition algebra $(\mathcal{C}, \cdot, \n)$ which consists of all $2 \times 2$ matrices of the following form:
    \begin{center}
    $\begin{pmatrix}
    a&u\\
    v&b\\
    \end{pmatrix}$, where $a, b \in \F, u, v \in \F^3$,
    \end{center}
    and
    \begin{align*}
    \begin{pmatrix}
    a&u\\
    v&b\\
    \end{pmatrix}
    \begin{pmatrix}
    a'&u'\\
    v'&b'\\
    \end{pmatrix}
    &=
    \begin{pmatrix}
    a a' + (u \, | \, v') & a u' + b' u - v \times v'\\
    a' v + b v' + u \times u' & b b' + (v \, | \, u')\\
    \end{pmatrix},\\
    \n\left(\begin{pmatrix}
    a&u\\
    v&b\\
    \end{pmatrix}
    \right)
    &= ab - (u \, | \, v),
    \end{align*}
    where $(\cdot \, | \, \cdot)$ and $\times$ denote the scalar and vector products of elements in $\F^3$, respectively.
    \end{definition}

    \begin{proposition}[{\cite[pp.~1040--1041]{Eld18}}]
    Let $\{ f_1, f_2, f_3 \}$ be the standard basis of~$\F^3$. Then the elements
    \[
    e_1 = \begin{pmatrix}
    1&0\\
    0&0\\
    \end{pmatrix},
    \quad
    e_2 = \begin{pmatrix}
    0&0\\
    0&1\\
    \end{pmatrix},
    \quad
    u_i = \begin{pmatrix}
    0&-f_i\\
    0&0\\
    \end{pmatrix},
    \quad
    v_i = \begin{pmatrix}
    0&0\\
    f_i&0\\
    \end{pmatrix},
    \quad
    i = 1,2,3,
    \]
    form the so-called {\em canonical basis} of the Zorn vector-matrix algebra $(\mathcal{C}, \cdot, \n)$. The multiplication table in this basis is given in Table~\ref{table:zorn-algebra}.
    \begin{table}[H]
    \centering
    $
    \begin{array}{|c||cc|ccc|ccc|}
    \hline
    \vphantom{\Big|} \cdot & e_1  & e_2 & u_1 & u_2& u_3 & v_1 & v_2 & v_3 \\\hline\hline
    \vphantom{\Big|} e_1 & e_1 & 0 & u_1 & u_2 & u_3 & 0 & 0 & 0 \\
    \vphantom{\Big|} e_2 & 0 & e_2 & 0 & 0 & 0 & v_1 & v_2 & v_3\\\hline
    \vphantom{\Big|} u_1 & 0 & u_1 & 0 & v_3 & -v_2 & -e_1 & 0 & 0\\
    \vphantom{\Big|} u_2 & 0 & u_2 & -v_3 & 0 & v_1 & 0 & -e_1 & 0\\
    \vphantom{\Big|} u_3 & 0 & u_3 & v_2 & -v_1 & 0 & 0 & 0 & -e_1\\\hline
    \vphantom{\Big|} v_1 & v_1 & 0 & -e_2 & 0 & 0 & 0 & u_3 & -u_2\\
    \vphantom{\Big|} v_2 & v_2 & 0 & 0 & -e_2 & 0 & -u_3 & 0 & u_1\\
    \vphantom{\Big|} v_3 & v_3 & 0 & 0 & 0 & -e_2 & u_2 & -u_1 & 0\\\hline
    \end{array}
    $
    \caption{\label{table:zorn-algebra} Multiplication table in a canonical basis of the Zorn vector-matrix algebra.}
    \end{table}
    \end{proposition}

    Note that an arbitrary Hurwitz algebra $(\mathcal{C}, \cdot, \n)$ is endowed with an involution which is called the {\em standard conjugation}: $\overline{x} = \n(1,x)1-x$. Then $x \cdot \overline{x} = \overline{x} \cdot x = \n(x)1$ for all $x \in \mathcal{C}$. In particular, in the case of the Zorn vector-matrix algebra, for any
    \[
    x = \begin{pmatrix}
    a&u\\
    v&b\\
    \end{pmatrix}
    \]
    it holds that $\n(1,x) = \tr(x) = a+b$, and then
    \[
    \overline{x} = \begin{pmatrix}
    b&-u\\
    -v&a\\
    \end{pmatrix}.
    \]

    \begin{definition}[{\cite[Definition~2.7]{Eld18}}]
        Let $(\mathcal{C}, \cdot, \n)$ be a Hurwitz algebra over a field~$\F$, and $\tau \in \Aut(\mathcal{C}, \cdot, \n)$ be an automorphism such that $\tau^3 = \id$. We define a new multiplication on $\mathcal{C}$ by the formula
        \begin{equation} \label{equation:petersson}
            x * y = \tau(\overline{x}) \cdot \tau^2(\overline{y}).
        \end{equation}
        Then the algebra $(\mathcal{C}, *, \n)$ is called a {\em Petersson algebra} and is denoted by $\mathcal{C}_{\tau}$.
    \end{definition}

    \begin{proposition}[{\cite[p.~1042]{Eld18}}] \label{proposition:reverse-Petersson}
    Let $(\mathcal{S}, *, \n)$ be a symmetric composition algebra, and $e \in \mathcal{S}$ be a nonzero idempotent. Then the mapping
    \begin{equation} \label{equation:tau-automorphism}
        \tau: x \mapsto e*(e*x) = \n(e,x)e - x*e
    \end{equation}
    is an automorphism of $(\mathcal{S}, *, \n)$, and $\tau^3 = \id$. Besides, one can define a new multiplication
    \begin{equation} \label{equation:Hurwitz-multiplication}
        x \cdot y = (e*x)*(y*e)
    \end{equation}
    on $\mathcal{S}$, and then $(\mathcal{S}, \cdot, \n)$ is a Hurwitz algebra with unity~$e$. Moreover, $\tau$ is also an automorphism of $(\mathcal{S}, \cdot, \n)$. The algebra $(\mathcal{S}, *, \n)$ is the Petersson algebra associated to the Hurwitz algebra $(\mathcal{S}, \cdot, \n)$ and its automorphism $\tau$, that is, Eq.~\eqref{equation:petersson} is satisfied.
    \end{proposition}

    \begin{proposition}[{\cite[Proposition~2.6]{Eld18}}]
    Let $(\mathcal{S}, *, \n)$ be a symmetric composition algebra, $e \in \mathcal{S}$ be a nonzero idempotent, and $\tau$ be the automorphism defined in Eq.~\eqref{equation:tau-automorphism}. Then the centralizer of the element~$e$:
    \[
    C_{(\mathcal{S}, *, \n)}(e) = \big\{ x \in \mathcal{S} \mid e*x=x*e \big\},
    \]
    coincides with the subalgebra of fixed points of the automorphism~$\tau$:
    \[
    \Fix(\tau) = \big\{ x \in \mathcal{S} \mid \tau(x) = x \big\}.
    \]
    \end{proposition}

    \begin{proposition}[{\cite[Proposition~9.9 and Theorem~9.13]{CEKT13}}]
    Let $\chrs \F = 3$, and $\Okubo$ be an Okubo algebra over~$\F$. Then $\Okubo$ is split if and only if $\Okubo$ contains a quaternionic idempotent (see Definition~\ref{definition:idempotents}). Moreover, such an idempotent is unique:
    \[
    e = z_{1,0} + z_{2,0} + z_{0,1} + z_{0,2} + z_{1,1} + z_{2,2} + z_{1,2} + z_{2,1}.
    \]
    \end{proposition}

    In \cite[Theorem~6.3]{Eld18} all possible cases of Petersson algebras associated to the Zorn vector-matrix algebra over a field of characteristic~$3$ are described in detail.
    We will only cite the first item of this statement.

    \begin{theorem}[{\cite[Theorem~6.3(i)]{Eld18}}] \label{theorem:tau-quaternionic}
    Let $(\mathcal{C},\cdot,\n)$ be an eight-dimensional Hurwitz algebra over a field~$\F$, $\chrs \F = 3$, and $\tau \in \Aut(\mathcal{C},\cdot,\n)$ be an order~$3$ automorphism. Then the norm $\n(\cdot)$ is isotropic, that is, $\mathcal{C}$ is isomorphic to the Zorn vector-matrix algebra over~$\F$.

    If, moreover, $(\tau - \id)^2 = 0$, then the Petersson algebra $\mathcal{C}_{\tau}$ is an Okubo algebra, the unity~$1$ of the algebra $(\mathcal{C},\cdot,\n)$ is its unique quaternionic idempotent~$e$, and there exists a canonical basis of~$\mathcal{C}$ such that
    \begin{equation} \label{equation:tau-canonical}
    \begin{aligned}
        \tau(e_1) &= e_1, \qquad& \tau(u_i) &= u_i, \quad i=1,2, \qquad&
        \tau(u_3) &= u_3 + u_2,\\
        \tau(e_2) &= e_2, \qquad& \tau(v_i) &= v_i, \quad i=1,3, \qquad&
        \tau(v_2) &= v_2 - v_3.
    \end{aligned}    
    \end{equation}
    In this case
    \[
    C_{(\mathcal{C}, *, \n)}(e) = \Fix(\tau) = \spn \{ e_1,e_2,u_1,u_2,v_1,v_3\}.
    \]
    \end{theorem}

    \begin{proposition}[{\cite[Proposition~24]{Eld15}, \cite[Theorem~8.5]{Eld18}}] \label{proposition:single-orbit}
    Let $\chrs \F = 3$, and $\Okubo$ be the split Okubo algebra over~$\F$. Then all singular idempotents in~$\Okubo$ belong to the same orbit under the action of the automorphism group. If $\F$ is algebraically closed, then all quadratic idempotents in~$\Okubo$ also belong to the same orbit.
    \end{proposition}

    \begin{proposition}[{\cite[Proposition~8.7]{Eld18}}] \label{proposition:idempotents-relation}
    Let $\chrs \F = 3$, $\Okubo$ be the split Okubo algebra over~$\F$, and $e$ be its quaternionic idempotent. Then
    \begin{enumerate}[{\rm (1)}]
        \item the set of singular idempotents in~$\Okubo$ coincides with
        \[
        \big\{ e + x \mid 0 \neq x \in C_{\Okubo}(e) \cap C_{\Okubo}(e)^{\perp} \big\};
        \]
        \item the set of all nonzero idempotents in~$\Okubo$ coincides with
        \[
        \big\{ e + x \mid x \in C_{\Okubo}(e), \: x*x = 0 \big\}.
        \]
    \end{enumerate}
    \end{proposition}

    \begin{lemma} \label{lemma:contained-in-centralizer}
    Let $\chrs \F = 3$, $\Okubo$ be the split Okubo algebra over~$\F$, and $e$ be its quaternionic idempotent. If $x*x = 0$ for some $x \in \Okubo$, then $x \in C_{\Okubo}(e)$.
    \end{lemma}

    \begin{proof}
    Let $\tau$ be the automorphism defined in Eq.~\eqref{equation:tau-automorphism} which corresponds to the idempotent~$e$. Consider a new multiplication on~$\Okubo$ given by Eq.~\eqref{equation:Hurwitz-multiplication}. Then, by Proposition~\ref{proposition:reverse-Petersson}, the Okubo algebra $(\Okubo, *, \n)$ is the Petersson algebra associated to the Hurwitz algebra $(\Okubo, \cdot, \n)$ with the unity~$e$ and its automorphism~$\tau$. It follows from~\cite[Theorem~6.3]{Eld18} that $\tau$ satisfies the conditions of Theorem~\ref{theorem:tau-quaternionic}. Then $(\Okubo, \cdot, \n)$ is isomorphic to the Zorn vector-matrix algebra over~$\F$, and there exists a canonical basis such that $\tau$ acts on basis vectors according to Eqs.~\eqref{equation:tau-canonical}. It remains to show that if $x*x = 0$ for some $x \in \Okubo$, then $x \in C_{\Okubo}(e) = \spn \{ e_1,e_2,u_1,u_2,v_1,v_3\}$.

    By Eq.~\eqref{equation:petersson}, 
    \[
    0 = x*x = \tau(\overline{x}) \cdot \tau^2(\overline{x}) = \tau(\overline{x} \cdot \tau(\overline{x})),
    \]
    hence $\overline{x} \cdot \tau(\overline{x}) = 0$. Let 
    \[
    \overline{x} = \sum_{i=1}^2 \lambda_i e_i + \sum_{i=1}^3 \mu_i u_i + \sum_{i=1}^3 \nu_i v_i.
    \]
    Then $\tau(\overline{x}) = \overline{x} + \mu_3 u_2 - \nu_2 v_3$. Since $\overline{x} = \tr(x) e - x$, we have
    \[
    0 = \overline{x} \cdot \tau(\overline{x}) = \overline{x} \cdot (\tr(x) e - x + \mu_3 u_2 - \nu_2 v_3) = \tr(x) \overline{x} - \n(x)e + \overline{x} \cdot (\mu_3 u_2 - \nu_2 v_3).
    \]
    By using multiplication table~\ref{table:zorn-algebra} and the equality $e = e_1 + e_2$, we obtain that the coefficients at $v_2$ and $u_3$ in the previous equation are equal to $\tr(x) \nu_2$ and $\tr(x) \mu_3$, respectively. Therefore, we have either $\nu_2 = \mu_3 = 0$ or $\tr(x) = 0$. In the second case we consider the coefficients at $v_1$ and $u_1$ which are equal to $-\mu_3^2$ and $-\nu_2^2$, respectively, and then $\nu_2 = \mu_3 = 0$ again. Consequently, $x \in C_{\Okubo}(e) = \spn \{ e_1,e_2,u_1,u_2,v_1,v_3\}$.
    \end{proof}

    \begin{notation}
    Let $x, y \in \Okubo$.
    We write $x \perp y$ if the elements $x$ and $y$ are orthogonal with respect to the symmetric bilinear form $\n(\cdot,\cdot)$, i.e., $\n(x,y) = 0$. 
    
    Given a subset $Y \subset \Okubo$, we also write $x \perp Y$ if $x \perp y$ for all $y \in Y$.
    \end{notation}

    \begin{corollary} \label{corollary:two-types-char-three}
    Let $\chrs \F = 3$, $\Okubo$ be the split Okubo algebra over~$\F$, and $e$ be its quaternionic idempotent. Consider an element $x \in Z(\Okubo)$ such that $x*x = 0$. Then there are two cases possible:
    \begin{enumerate}[{\rm (1)}]
        \item $x \perp C_{\Okubo}(e)$, $e + x$ is a singular idempotent, and for any $y \in O_{\Okubo}(x)$ we have $y*y = 0$;
        \item $x \not\perp C_{\Okubo}(e)$, $e + x$ is a quadratic idempotent, and all elements $y \in O_{\Okubo}(x)$ such that $y*y = 0$ are contained in a single two-dimensional subspace.
    \end{enumerate}
    All elements of the first type belong to the same orbit under the action of the automorphism group. If $\F$ is algebraically closed, then all elements of the second type also belong to the same orbit.
    \end{corollary}

    \begin{proof}
    It follows from Lemma~\ref{lemma:contained-in-centralizer} that $x \in C_{\Okubo}(e)$. By Proposition~\ref{proposition:idempotents-relation}, $e + x$ is a singular idempotent if $x \perp C_{\Okubo}(e)$, and a quadratic idempotent otherwise. By Proposition~\ref{proposition:single-orbit}, all singular idempotents in~$\Okubo$ belong to the same orbit under the action of the automorphism group. Since the quaternionic idempotent
    \[
    e = z_{1,0} + z_{2,0} + z_{0,1} + z_{0,2} + z_{1,1} + z_{2,2} + z_{1,2} + z_{2,1}
    \]
    is unique, it is preserved by any automorphism of~$\Okubo$. Therefore, the elements~$x$ of type~(1) indeed form an orbit under the action of the automorphism group.

    By using the fact that $\dim C_{\Okubo}(e) = 6$, one can easily obtain
    \begin{align*}
    C_{\Okubo}(e) = \spn & \{ z_{1,0} + z_{2,0}, z_{0,1} + z_{0,2}, z_{1,1} + z_{2,2}, z_{1,2} + z_{2,1}, \\
    & (z_{0, 2} + z_{1, 2} + z_{2, 2}) - (z_{0, 1} + z_{2, 1} + z_{1, 1}), \\
    & (z_{1,0} + z_{0,1} + z_{2,2}) - (z_{2,0} + z_{0,2} + z_{1,1})\}. 
    \end{align*}
    Consider $x = (z_{0, 2} + z_{1, 2} + z_{2, 2}) - (z_{0, 1} + z_{2, 1} + z_{1, 1})$. Since $x \perp C_{\Okubo}(e)$, the element $e + x$ is a singular idempotent (see also an example after Definition~22 in~\cite{Eld15} and~\cite[Theorem~3.12]{Eld09}). It is shown in Lemma~\ref{lemma:orthogonality-characteristic-three} that $y*y = 0$ for all $y \in O_{\Okubo}(x)$. Clearly, this property of the orthogonalizer is preserved under automorphism, so it holds for all elements of type~(1).
    
    We now consider elements of type~(2). Passing to the algebraic closure $\overline{\F}$ of~$\F$, we obtain the split Okubo algebra $\Okubo_{\overline{\F}} = \Okubo \otimes_{\F} \overline{\F}$ over $\overline{\F}$. The element $e \in \Okubo_{\overline{\F}}$ is again a quaternionic idempotent, and all quadratic idempotents of~$\Okubo$ are also quadratic idempotents of $\Okubo_{\overline{\F}}$. By Proposition~\ref{proposition:single-orbit}, all quadratic idempotents in~$\Okubo_{\overline{\F}}$ belong to the same orbit under the action of the automorphism group.
    Thus, over an algebraically closed field, all elements of type~(2) belong to the same orbit.

    Consider $x = z_{0,2} + z_{1,2} + z_{2,2} \in \Okubo_{\overline{\F}}$. Since $x \not\perp C_{\Okubo}(e)$, the element~$x$ is of type~(2). According to Example~\ref{example:length-two-path}, all elements $y \in O_{\Okubo_{\overline{\F}}}(x)$ such that $y*y = 0$ are contained in the two-dimensional subspace $\spn \{ z_{0,2} + z_{1,2} + z_{2,2}, z_{0,1} + z_{1,1} + z_{2,1} \}$. Therefore, for an arbitrary element~$x \in \Okubo_{\overline{\F}}$ of type~(2), all elements $y \in O_{\Okubo_{\overline{\F}}}(x)$ such that $y*y = 0$ are contained in a single two-dimensional subspace.
    
    Let now $x \in \Okubo$ be of type~(2). By Lemma~\ref{lemma:zero-square-subspace} and Corollary~\ref{corollary:field-condition}, there exists a two-dimensional subspace $W \subset O_{\Okubo}(x)$ such that $y*y = 0$ for all $y \in W$. Assume that there exists $z \in O_{\Okubo}(x) \setminus W$ which satisfies $z*z = 0$. Then $x \in \Okubo_{\overline{\F}}$ is also of type~(2), $y*y = 0$ for any $y \in W_{\overline{\F}} = W \otimes_{\F} \overline{\F} \subset O_{\Okubo_{\overline{\F}}}(x)$, and $z \in O_{\Okubo_{\overline{\F}}}(x) \setminus W_{\overline{\F}}$. This is a contradiction to the last sentence of the previous paragraph. Thus all elements $y \in O_{\Okubo}(x)$ such that $y*y = 0$ are contained in a single two-dimensional subspace.
    \end{proof}

    \begin{proposition} \label{proposition:zero-square-class}
    The elements which satisfy the conditions of Corollary~\ref{corollary:two-types-char-three}(1) are exactly the nonzero elements of the set
    \[
    \spn \{ (z_{0, 2} + z_{1, 2} + z_{2, 2}) - (z_{0, 1} + z_{2, 1} + z_{1, 1}), \; (z_{1,0} + z_{0,1} + z_{2,2}) - (z_{2,0} + z_{0,2} + z_{1,1})\}, 
    \]
    and any two of them are orthogonal to each other.
    \end{proposition}

    \begin{proof}
    It is shown in Lemma~\ref{lemma:orthogonality-characteristic-three} that the element $x = (z_{0, 2} + z_{1, 2} + z_{2, 2}) - (z_{0, 1} + z_{2, 1} + z_{1, 1})$ satisfies the conditions of Corollary~\ref{corollary:two-types-char-three}(1). Consider similarly the element $y = (z_{1,0} + z_{0,1} + z_{2,2}) - (z_{2,0} + z_{0,2} + z_{1,1}) \perp C_{\Okubo}(x)$. In terms of Lemma~\ref{lemma:orthogonality-characteristic-three}, we have $y = -w_k$ for $\beta = 1$ and $k = -1$. Thus $y*y = 0$, and $y$ is orthogonal to $x$. Alternatively, one can note that $y = \varphi(x)$ for $\varphi \in \Aut(\Okubo)$ given by
    \begin{align*}
        \varphi(z_{1,0}) &= z_{2,1}, & \varphi(z_{0,1}) &= z_{1,1}, & \varphi(z_{1,1}) &= z_{0,2}, & \varphi(z_{1,2}) &= z_{1,0},\\
        \varphi(z_{2,0}) &= z_{1,2}, & \varphi(z_{0,2}) &= z_{2,2}, & \varphi(z_{2,2}) &= z_{0,1}, & \varphi(z_{2,1}) &= z_{2,0}.
    \end{align*}
    Since $x*y=y*x=0$ and $x, y \perp C_{\Okubo}(e)$, any linear combination of $x$ and $y$ also satisfies the conditions of Corollary~\ref{corollary:two-types-char-three}(1). It remains to note that $C_{\Okubo}(e) \cap C_{\Okubo}(e)^{\perp}$ coincides with $\spn \{ u_2, v_3 \}$ (see the proof of Proposition~8.7 in~\cite{Eld18}) and has dimension~$2$.
    \end{proof}
    
    \begin{corollary} \label{corollary:unique-path-three}
    Let $\Okubo$ be the split Okubo algebra over a field~$\F$, $\chrs \F = 3$. Then all shortest paths in $\Gamma_O(\Okubo)$ are unique.
    \end{corollary} 

    \begin{proof}
    Assume first that $\F$ is algebraically closed. Consider an arbitrary element which satisfies the conditions of Corollary~\ref{corollary:two-types-char-three}(2). Then it belongs to the same orbit under the action of the automorphism group as the element $z_{0,2} + z_{1,2} + z_{2,2}$ whose orthogonalizer contains $(z_{0, 2} + z_{1, 2} + z_{2, 2}) - (z_{0, 1} + z_{2, 1} + z_{1, 1})$. Thus any element of type~(2) is adjacent to some element of type~(1). By Proposition~\ref{proposition:zero-square-class}, all elements of type~(1) are orthogonal to each other. It follows that if $x*x = y*y = 0$ and $\D([x],[y]) = 3$ for some $x, y \in Z(\Okubo)$, then both elements $x$ and $y$ are of type~(2). Hence, by Proposition~\ref{proposition:m-paths-length-three}, there exists a unique shortest path between $[x]$ and $[y]$.

    Let now $\F$ be an arbitrary field with $\chrs \F = 3$, and $x, y \in Z(\Okubo)$ be such that $x*x = y*y = 0$ and $\D([x],[y]) = 3$ in $\Gamma_O(\Okubo)$. Passing to the algebraic closure $\overline{\F}$ of~$\F$, we obtain the split Okubo algebra $\Okubo_{\overline{\F}} = \Okubo \otimes_{\F} \overline{\F}$ over $\overline{\F}$. By Corollary~\ref{corollary:distance-at-most-two}, it holds that $\n(x,y) \neq 0$, so in $\Gamma_O(\Okubo_{\overline{\F}})$ we again have $\D([x],[y]) = 3$. Any shortest path between $[x]$ and $[y]$ in $\Gamma_O(\Okubo)$ is also a shortest path in $\Gamma_O(\Okubo_{\overline{\F}})$, so there is only one such path.

    If $x*x = y*y = 0$ and $\D([x],[y]) = 2$ in $\Gamma_O(\Okubo)$, then the shortest path between $[x]$ and~$[y]$ is unique by Lemma~\ref{lemma:length-two-path}. Since $\Gamma_O(\Okubo)$ is a simple graph, any shortest path of length one is unique. Hence the uniqueness of the shortest path between an arbitrary pair of vertices $[x]$ and~$[y]$ with $\n(x,x*x) = \n(y,y*y) = 0$ follows from the proof of Lemma~\ref{lemma:diameter-five}.
    
    It remains to apply Proposition~\ref{proposition:connectivity-component-with-type3-elem} in the case when $\n(x,x*x) \neq 0 \neq \n(y,y*y)$.
    \end{proof}
	
	\section{The case of a field containing a primitive cube root of unity}

    In this section we pay special attention to the orthogonality graph of an Okubo algebra with nonzero idempotents in the case of a field~$\F$, $\chrs \F \neq 3$, which contains a primitive cube root of unity~$\omega$. Recall that, by Proposition~\ref{proposition:omega-isotropic}(2), such an Okubo algebra is always isomorphic to the pseudo-octonion algebra $P_8(\F)$, that is, it can be considered as the set $\frsl_3(\F)$ of traceless $3 \times 3$ matrices over~$\F$ with the multiplication given by Eq.~\eqref{equation:product}. In particular, if the field~$\F$ is algebraically closed, then by definition any Okubo algebra over~$\F$ is isomorphic to $P_8(\F)$, so it always contains a nonzero idempotent (see also~\cite[Theorem~2.8]{EM91} which shows that any finite-dimensional flexible composition algebra over an algebraically closed field with $\chrs \F \neq 2, 3$ contains a nonzero idempotent).
    
    In order to emphasize the connection between the orthogonality graph of $P_8(\F)$ and orthogonality graphs of matrix rings, we give an alternative proof of Theorem~\ref{theorem:graph-of-orthogonality} which relies on the results of the paper~\cite{Guterman-Markova_Orthogonality} by Bakhadly, Guterman, and Markova.

    According to Proposition~\ref{proposition:norm-of-zero-divisor}, the criterion $\n(a) = 0$ characterizes zero divisors in an arbitrary Okubo algebra. However, in the case of the pseudo-octonion algebra $P_8(\F)$, a complete and explicit classification of zero divisors can be given.

    \begin{proposition} \label{proposition:zero-divisors}
        Let $\chrs \F \neq 3$ and $\omega \in \F$. Then 
	\[ Z(P_8(\F)) = (\mathcal{N}_3 \cup \F \Orb(\Omega)) \setminus \{ 0 \},\]
    where $\mathcal{N}_3$ is the set of nilpotent matrices in $M_3(\F)$, and $\Orb(\Omega)$ is the orbit of the element
        \begin{equation} \label{equation:matrix-x}
        \Omega = \begin{pmatrix}
			1 &  0 & 0 \\
			0  & \omega & 0 \\
			0  &  0 & \omega^2
		\end{pmatrix}
        \end{equation}
        under the conjugation action of $GL_3(\F)$ on $P_8(\F)$.
	\end{proposition}

    \begin{proof}
    Let $x \in M_3(\F)$ be an arbitrary traceless matrix. It follows from the Cayley--Hamilton theorem and Eq.~\eqref{equation:ch} that
        \[ x^3 + \s(x)x - \det(x) I = 0. \]
	Assuming that $\n(x) = 0$, we use Eq.~\eqref{equation:norm-division-by-3} to obtain $\s(x) = 0$, and hence
	\begin{equation} \label{equation:CH-for-zero-norm-elem}
		x^3 = \det(x)I.
	\end{equation}
	
	Consider two cases:
	\begin{enumerate}[\rm (1)]
		\item If $\det(x) = 0$, then $x^3 = 0$, so $x \in \mathcal{N}_3$. On the other hand, it can be easily seen that $\mathcal{N}_3 \subseteq P_8(\F)$ and all elements in $\mathcal{N}_3$ have zero norm.
		\item Let now $\det(x) \neq 0$. We pass to the algebraic closure $\overline{\F}$ of the field~$\F$ and note that if $\Jord_x$ is the Jordan normal form of the matrix~$x$ over~$\overline{\F}$, then $\Jord_x \in P_8(\overline{\F})$, and the condition $\n(x) = 0$ is equivalent to $\n(\Jord_x) = 0$. Hence Eq.~\eqref{equation:CH-for-zero-norm-elem} holds also for~$\Jord_x$, and $\det(\Jord_x) \neq 0$. One can easily verify that $\Jord_x$ must be proportional to the matrix $\Omega$ defined in Eq.~\eqref{equation:matrix-x}. Since $\omega \in \F$, the matrix $J_x$ is also the Jordan normal form of~$x$ over the field~$\F$. \qedhere
	\end{enumerate}
    \end{proof}

    \begin{remark} \label{remark:difference-btw-zero-divisors}
    If $y \in Z(P_8(\F))$, then it can be easily seen from Eq.~\eqref{equation:product} that $y*y = yy$. It follows from Eq.~\eqref{equation:bilinear-form} that $\n(y, y*y) = \n(y, yy) = \tr(y^3)/3$.
        \begin{itemize}
    		\item For $y \in \F \Orb(\Omega)$ we have $y = \gamma g\Omega g^{-1}$, where $g \in GL_3(\F)$ and $0 \neq \gamma \in \F$. Hence $\tr(y^3) = \gamma^3 \tr(\Omega^3) = \gamma^3 \tr(I) \neq 0$, and thus $\n(y, y*y)\neq 0$.
    		\item For a nilpotent matrix $y \in \mathcal{N}_3$ we have $y^3 = 0$, so $\tr(y^3) = 0$ and $\n(y, y*y) = 0$.
    	\end{itemize}
    \end{remark}
	
	\begin{corollary} \label{corollary:prelemma}
		Let $a, \, b \in \mathcal{N}_3$. Then the following conditions are equivalent:
		\begin{enumerate} [\rm (1)]
			\item $a*b = b*a = 0$, \label{enumerate:ort1}
			\item $ab = ba = 0$. \label{enumerate:ort2}
		\end{enumerate}
	\end{corollary}
	
	\begin{proof}
    Assume that condition~\eqref{enumerate:ort1} is satisfied. By Remark~\ref{remark:difference-btw-zero-divisors}, it holds that $\n(a, a*a) = \n(b, b*b) = 0$. We now show that $\n(a,b) = 0$. If $a * a \neq 0$, then, by Corollary~\ref{corollary:orthogonalizer-formula}, we have $b \in \F(a*a)$, and the desired equality follows. If $a*a = 0$, then $b = a * x$ for some $x \in a^{\perp}$. Since $P_8(\F)$ is a symmetric composition algebra, we have $\n(a,b) = \n(a,a*x) = \n(a*a,x) = \n(0,x) = 0$.
    
    Then it follows from Eq.~\eqref{equation:bilinear-form} that $\tr(ab) = \tr(ba) = 0$. By using Definition~\eqref{equation:product} of the multiplication~``$*$'', we obtain $ab = ba = 0$.
    
    Conversely, condition~\eqref{enumerate:ort2} guarantees~\eqref{enumerate:ort1} due to the definition of the multiplication~``$*$''.
	\end{proof}
	
    An alternative proof of the main result on the orthogonality graph $\Gamma_O(P_8(\F))$ relies on the statement of the following theorem:
	
	\begin{theorem}[{\cite[Theorem~5.9]{Guterman-Markova_Orthogonality}}] \label{theorem:main-for-nilp-in-mat-alg}
		Let $\mathbb{F}$ be an arbitrary field, $n \geq 3$, and let $\mathcal{N}_n$ denote the set of all nilpotent matrices in $M_n(\F)$. We also denote the subgraph of $\Gamma_O(M_n(\F))$ on the vertex set $\mathbb{P}(\mathcal{N}_n)$ by $\Gamma_O^{\mathcal{N}_n}(M_n(\F))$. Then 
        \[
        \diam(\Gamma_O^{\mathcal{N}_n}(M_n(\F))) = 
        \begin{cases}
            5, & n = 3,\\
            4, & n \geq 4.
        \end{cases}
        \]
	\end{theorem}
	
	\begin{proof}
    In~\cite{Guterman-Markova_Orthogonality} the orthogonality graph of the algebra~$\A$ is defined so that its vertices are all two-sided zero divisors of~$\A$, whereas in our case the vertices correspond to proportionality classes of these elements. Still, the above-mentioned result on the diameter of $\Gamma_O^{\mathcal{N}_n}(M_n(\F))$ remains valid. Indeed,  
    \begin{itemize}
        \item each path $[v_0] \leftrightarrow \dots \leftrightarrow [v_m]$ in our graph corresponds to the path $v_0 \leftrightarrow \dots \leftrightarrow v_m$ in the graph from~\cite{Guterman-Markova_Orthogonality};
        \item any shortest path $v_0 \leftrightarrow \dots \leftrightarrow v_m$ in the graph from~\cite{Guterman-Markova_Orthogonality} induces a path $[v_0] \leftrightarrow \dots \leftrightarrow [v_m]$ in our graph. Note that the loops are prohibited in our definition of the orthogonality graph. Thus, if $[v_i] = [v_{i+1}]$, then we ``glue together'' this pair of adjacent vertices, and the length of the path decreases by~$1$. However, if $m \geq 3$, then all the classes $[v_i]$ are pairwise distinct, as otherwise the original path could be shortened.
    \end{itemize}
    Thus, the difference between the definitions does not affect the diameter of the orthogonality graph if its value is at least~$3$.
	\end{proof}
	
	\begin{lemma} \label{lemma:subgraph-of-nilp}
		Let $\chrs \F \neq 3$ and $\omega \in \F$. Then the graphs $\Gamma_O^{\mathcal{N}_3}(P_8(\F))$ and $\Gamma_O^{\mathcal{N}_3}(M_3(\F))$ are isomorphic.
	\end{lemma}
	
	\begin{proof}
		Follows immediately from Corollary~\ref{corollary:prelemma}.
	\end{proof}
	
	\begin{theorem} \label{theorem:graph-of-orthogonality-main}
		Let $\chrs \F \neq 3$, and $\F$ contains a primitive cube root of unity~$\omega$. The graph $\Gamma_O(P_8(\F))$ is disconnected, and the vertex sets of its connected components are as follows:
		\begin{enumerate} [\rm (1)]
			\item the set
			\[ V_{x} = \big\{ [x], \, [x*x] \big\} \]
            for any $x \in \Orb(\Omega)$, where $\Orb(\Omega)$ is the orbit of the element $\Omega$ defined in Eq.~\eqref{equation:matrix-x} under the conjugation action of $GL_n(\F)$ on $P_8(\F)$;
			\item the set
			\[ V = \big\{ [x] \mid x \in \mathcal{N}_3 \setminus \{ 0 \} \big\}, \]
            where $\mathcal{N}_3$ is the set of all nilpotent matrices in $M_3(\F)$.
		\end{enumerate}
		The diameter of the connected component on the vertex set $V_x$ equals $1$, and on the vertex set $V$ --- equals $5$.
		Note that the sets $V_x$ can coincide for distinct values of~$x$, and in this case they define the same connected component.
	\end{theorem}
	
	\begin{proof}
		By Proposition~\ref{proposition:zero-divisors}, zero divisors in~$P_8(\F)$ are exactly nonzero elements from the union of $\mathcal{N}_3$ and $\F \Orb(\Omega)$.
        
        Consider first an arbitrary element $x \in \Orb(\Omega)$. According to Remark~\ref{remark:difference-btw-zero-divisors}, we have $\n(x, x*x) \neq 0$, so it follows from Proposition~\ref{proposition:connectivity-component-with-type3-elem} that $V_x$ forms a connected component in $\Gamma_O(P_8(\F))$.
		
	By Theorem~\ref{theorem:main-for-nilp-in-mat-alg} and Lemma~\ref{lemma:subgraph-of-nilp}, the vertices $\mathbb{P}(\mathcal{N}_3)$ of the graph $\Gamma_O(P_8(\F))$, which correspond to all zero divisors remaining, form a connected component whose diameter equals~$5$.
	\end{proof}

    \begin{corollary} \label{corollary:nilpotent-bigeodetic}
    Consider a field~$\F$ such that $\chrs \F \neq 3$ and $\omega \in \F$. Let nonzero matrices $x, y \in \mathcal{N}_3$ be such that $\D([x],[y]) = k$ in $\Gamma_O^{\mathcal{N}_3}(M_3(\F))$, $0 \leq k \leq 5$. Then either there exists a unique path of length $k$ between $[x]$ and $[y]$, or there are exactly two such paths. The second possibility is achieved precisely in the following cases:
    \begin{enumerate}[{\rm (1)}]
        \item $\D([x],[y]) = 3$, and $x^2 = y^2 = 0$;
        \item $\D([x],[y]) = 4$, and $x^2 = 0$ or $y^2 = 0$;
        \item $\D([x],[y]) = 5$.
    \end{enumerate}
    \end{corollary}

    \begin{proof}
        In view of Lemma~\ref{lemma:subgraph-of-nilp} and the equality $z*z = zz$ for all $z \in Z(P_8(\F))$, Corollary~\ref{corollary:unique-path} on the number of shortest paths between any two vertices in $\Gamma_O(P_8(\F))$ can be reformulated in terms of the graph $\Gamma_O^{\mathcal{N}_3}(M_3(\F))$.
    \end{proof}

    \medskip

    The authors are grateful to Professor Alexander E. Guterman for careful attention to the work and fruitful discussions.

    % \newpage

    \appendix

    \section{A program for multiplication and norm computation} \label{section:appexdix-program}

     By Theorem~\ref{theorem:isotropic-norm}, an Okubo algebra over an arbitrary field~$\F$ has isotropic norm if and only if it is isomorphic to $\Okubo_{\alpha,\beta}$ for some $\alpha, \beta \in \F \setminus \{ 0 \}$. We use Table~\ref{table:okubo-algebra-isotropic} and Eq.~\eqref{equation:scpr-of-basis} to write a Wolfram Mathematica program for multiplication and norm computation in $\Okubo_{\alpha,\beta}$. The program was tested in Wolfram Mathematica 11.0. 
     
     Below we define the functions \verb|mult[x,y]| and \verb|n[x,y]| which take two elements $x$ and~$y$ as their arguments and return $xy$ and $\n(x,y)$, respectively. The parameters \verb|a| and \verb|b| correspond to $\alpha$ and $\beta$, respectively. The basis vectors of $\Okubo_{\alpha,\beta}$ are considered as standard basis vectors in an eight-dimensional space over~$\F$, and \verb|zo| is a zero vector of length~$8$, that is, the zero element of $\Okubo_{\alpha,\beta}$.

    \begin{lstlisting}
basis[i_] := Join[Table[0, i - 1], {1}, Table[0, 8 - i]];
{z10, z20, z01, z02, z11, z22, z12, z21} = Table[basis[i], {i, 8}];
zo = Table[0, 8];
multtable =
   {{z20, zo, -z11, zo, -z21, zo, zo, -a z01},
    {zo, a z10, zo, -z22, zo, -a z12, -a z02, zo},
    {zo, -z21, z02, zo, zo, -b z20, zo, -z22},
    {-z12, zo, zo, b z01, -b z10, zo, -b z11, zo},
    {zo, -a z01, -z12, zo, z22, zo, -b z20, zo},
    {-a z02, zo, zo, -b z21, zo, a b z11, zo, -a b z10},
    {-z22, zo, -b z10, zo, zo, -a b z01, b z21, zo},
    {zo, -a z11, zo, -b z20, -a z02, zo, zo, a z12}};
mult[x_, y_] := Sum[(x.multtable)[[i]]*y[[i]], {i, 8}]
ntable =
   {{0, a, 0, 0, 0, 0, 0, 0},
    {a, 0, 0, 0, 0, 0, 0, 0},
    {0, 0, 0, b, 0, 0, 0, 0},
    {0, 0, b, 0, 0, 0, 0, 0},
    {0, 0, 0, 0, 0, a b, 0, 0},
    {0, 0, 0, 0, a b, 0, 0, 0},
    {0, 0, 0, 0, 0, 0, 0, a b},
    {0, 0, 0, 0, 0, 0, a b, 0}};
n[x_, y_] := x.ntable.y
    \end{lstlisting}

    For instance, in order to verify calculations from Example~\ref{example:distance-five}, we write:
    \begin{lstlisting}
a = 1;
x = z01 - z11;
y = z02 - z22;
xx = mult[x,x]
yy = mult[y,y]
n[x,xx]
n[y,yy]
n[xx,yy]
    \end{lstlisting}
    The program returns:
    \begin{lstlisting}
{0, 0, 0, 1, 0, 1, 1, 0}
{0, 0, b, 0, b, 0, 0, b}
0
0
3b^2
    \end{lstlisting}
    Indeed, we have $x*x = z_{0, 2} + z_{1, 2} + z_{2, 2}$, $y*y = \beta (z_{0, 1} + z_{2, 1} + z_{1, 1})$, $\n(x, x*x) = \n(y, y*y) = 0$, and $\n(x*x, y*y) = 3\beta^2$.

    % \newpage
    
    \bibliographystyle{tugboat}
    \bibliography{bibfile}
    
	\end{document}